\documentclass[a4paper, 10pt]{article}
\usepackage[latin1]{inputenc}
\usepackage{amsthm}
\usepackage{amsmath,amsfonts}
\usepackage{color}
\usepackage{subfigure}
\usepackage{graphicx}
\usepackage{bm}
\usepackage{bbm}
\usepackage{epstopdf}
\usepackage{indentfirst}
\usepackage[margin=0.8in]{geometry}
\newtheorem{theorem}{Theorem}[section]
\newtheorem{proposition}{Proposition}[section]
\newtheorem{definition}{Definition}[section]
\newtheorem{lemma}{Lemma}[section]
\newtheorem{notation}{Notation}[section]
\newtheorem{remark}{Remark}[section]

\numberwithin{equation}{section} 

\newcommand\ds{\displaystyle}

\def\e{\varepsilon}
\def\d{\delta}

\def\O{\Omega}

\def\R{{\mathbb{R}}}
\def\Z{{\mathbb{Z}}}
\def\V{{\mathbb{V}}}

\begin{document}
\begin{center}
{\bf \large Homogenization via unfolding in domains separated by the thin layer of the thin beams}
\end{center} 
  \centerline{Georges Griso, Anastasia Migunova, Julia Orlik}

\begin{abstract}
We consider a thin heterogeneous layer consisted of the thin beams (of radius $r$) and we study the limit behavior of this problem as the periodicity $\e$, the thickness $\delta$ and the radius $r$ of the beams tend to zero. The decomposition of the displacement field in the beams developed in \cite{griso} is used, which allows to obtain a priori estimates. Two types of the unfolding operators are introduced to deal with the different parts of the decomposition. In conclusion we obtain the limit problem together with the transmission conditions across the interface.
\end{abstract}

\section{Introduction}
In this paper a system of elasticity equations in the domains separated by a thin heterogeneous layer is considered. The layer is composed of periodically distributed vertical thin, compared to their length, beams, whose diameter and height tend to zero together with the period of the structure. The structure is clamped on the bottom. We consider the case of an isotropic linearized elasticity system.

The elasticity problems involving thin layers of periodic fiber--networks appear in many technical applications, where special constraints on stiffness of technical textiles or composites are required, depending on a type of the application. For example, drainage and protective wear, working for outer--plane compression, should provide certain stiffness against external mechanical loading. 

Thin layers were considered in number of papers (see e.g. \cite{onofr, nrj, geym, past, zhpast, pan}). In particular, \cite{onofr} deals with a layer composed of the holes scaled with additional small parameter; \cite{nrj, geym} consider a case of a soft layer, whose stiffness is scaled by the thickness of the layer. Thin beams and their junction with 3D structures were  studied in \cite{griso, curved, griso2, griso3}: \cite{griso} deals with the homogenization of a single thin body; in \cite{curved} a structure made of these bodies is considered. \cite{griso2}, \cite{griso3} study the limit behavior of structures composed of rods in junction with a plate.

Our problem contains  3 small parameters: the thickness $\d$ of the layer (and the height of the beams at the same time), the radius $r$ of the rods and the period of the layer $\e$. Obviously, the choice of an appropriate scaling of the problem defines what limit will be reached. For example, in \cite{past, zhpast, pan} 3D periodic fiber--networks were considered and it is investigated, that if $r$ is of the order $\e^2$ the bending moments in beams enter the homogenized macroscopic equation as micro--polar rotational degrees of freedom.

Considering the structure made of thin beams the first difficulty arises when we obtain estimates on the displacements. To overcome this problem, we use the decomposition of thin beam's displacement into a displacement of a mean line and a rotation of its cross--section, introduced in \cite{griso}. After deriving the estimates on the decomposed components, we get bounds for the minimizing sequence which depend on $\e, r, \d$. 3 critical cases were obtained with different ratios between small parameters. Two of them are considered in the present paper and lead to the same kind of the limit problem. The third one corresponds no longer to the thin beams but to the small inclusions and therefore is not studied in the present paper. 

In order to obtain the limit problem, periodic unfolding method,\cite{unfold}, is applied to the components of the decomposition. Two additional types of unfolding operators are introduced to deal with the mean displacements and rotations, depending only on component $x_3$, and the warping depending on all $(x_1, x_2, x_3)$. In the limit, a 3D elasticity problem for two domains is obtained, where the domains are separated by the interface with an inhomogeneous Robin--type condition. The coefficients in the Robin condition are obtained from an auxiliary 1D bending problem for a beam. An important result is that the displacements are continuous in a direction normal to the interface and have a jump only in a tangential direction.

The paper is organized as follows. In Section 2, geometry and weak and strong formulations of the problem are introduced. Section 3 presents decomposition of a single beam and the preliminary estimates. Section 4 is devoted to derivation of a priori estimate in all subdomains of $\Omega_{r, \e, \d}$. In Section 5, the periodic unfolding operators are introduced and their properties are defined. Also the limit fields for the beams based on the estimates from  Section 4 are defined. Section 6 deals with passing to the limit and obtaining the variational formulation for the limit problem. In Section 7, the results are summarized: the strong formulation for the limit problem is given and the final result on the convergences of the solutions is introduced. Section 8 contains additional information. Section 9 provides an auxiliary lemma, used in the proofs.

\section{The statement of the problem}
\subsection{Geometry}
In the Euclidean space $\R^2$ let $\omega$ be a connected domain with Lipschitz boundary and let $L > 0$ be a fixed real number. Define the reference domains: 
\begin{equation*}
\begin{aligned}
&\Omega^- &= &\, \omega \times ( -L, 0 ),&\\
&\Omega^+ &= &\, \omega \times ( 0, L ),&\\
&\Sigma &= &\, \omega \times \{ 0 \}.&
\end{aligned}
\end{equation*}
Moreover, $\Omega$ (see Figure \ref{limfig}b) is defined by 
\begin{equation}
\Omega = \Omega^+ \cup \Omega^- \cup \Sigma = \omega \times (-L, L).
\end{equation}

For the domains corresponding to the structure with the layer of thickness $\delta$ introduce the following notations:
\begin{equation*}
\begin{aligned}
&\Omega^+_{\delta} &= &\, \omega \times ( \delta, L ),&\\
&\Sigma^+_{\delta} &= &\, \omega \times \{ \delta \}.&
\end{aligned}
\end{equation*}

In order to describe the configuration of the layer, for any  $(d,r)\in (0,+\infty)^2$  we define the rod $B_{r,d}$ by
$$
B_{r, d} = D_r \times (0, d)$$ where $D_r = D(O , r)$ is the disc of center $O$ and radius $r$.
\smallskip

The set of  rods is
\begin{equation}
\Omega_{r, \e,\delta}^i = \bigcup\limits_{{\bf i}\in \widehat{\Xi}_\e\times\{0\}} \left\{ x\in \R^3\;\;|\;\; x \in {\bf i}\varepsilon+B_{r, \delta}\right\},
\end{equation}
where
\begin{equation}
\widehat{\Xi}_\varepsilon=\Big\{ \xi \in {\mathbb Z}^2\;\; |\; \varepsilon(\xi +Y)\subset \omega\Big\}, \quad Y =\left( -\frac{1}{2}; \frac{1}{2}\right)^2.
\end{equation}
Moreover, we set:
\begin{equation}
\widehat{\omega}_\e = \hbox{interior}\bigcup_{{\bf i}\in \widehat{\Xi}_\e}\e\big({\bf i}+\overline{Y}\big).
\end{equation}

The physical reference configuration (see Figure \ref{limfig}a) is defined by $\Omega_{r, \e,\delta}$:
\begin{equation}
\Omega_{r, \e,\delta}=\hbox{interior}\Big(\overline{\Omega^-}\cup\overline{\Omega_{r, \e,\delta}^i }\cup \overline{\Omega^+_\delta }\Big).
\end{equation}

\begin{figure}[ht!]
\centering
\subfigure[The domain with the thin layer]{\includegraphics[width=0.43\textwidth]{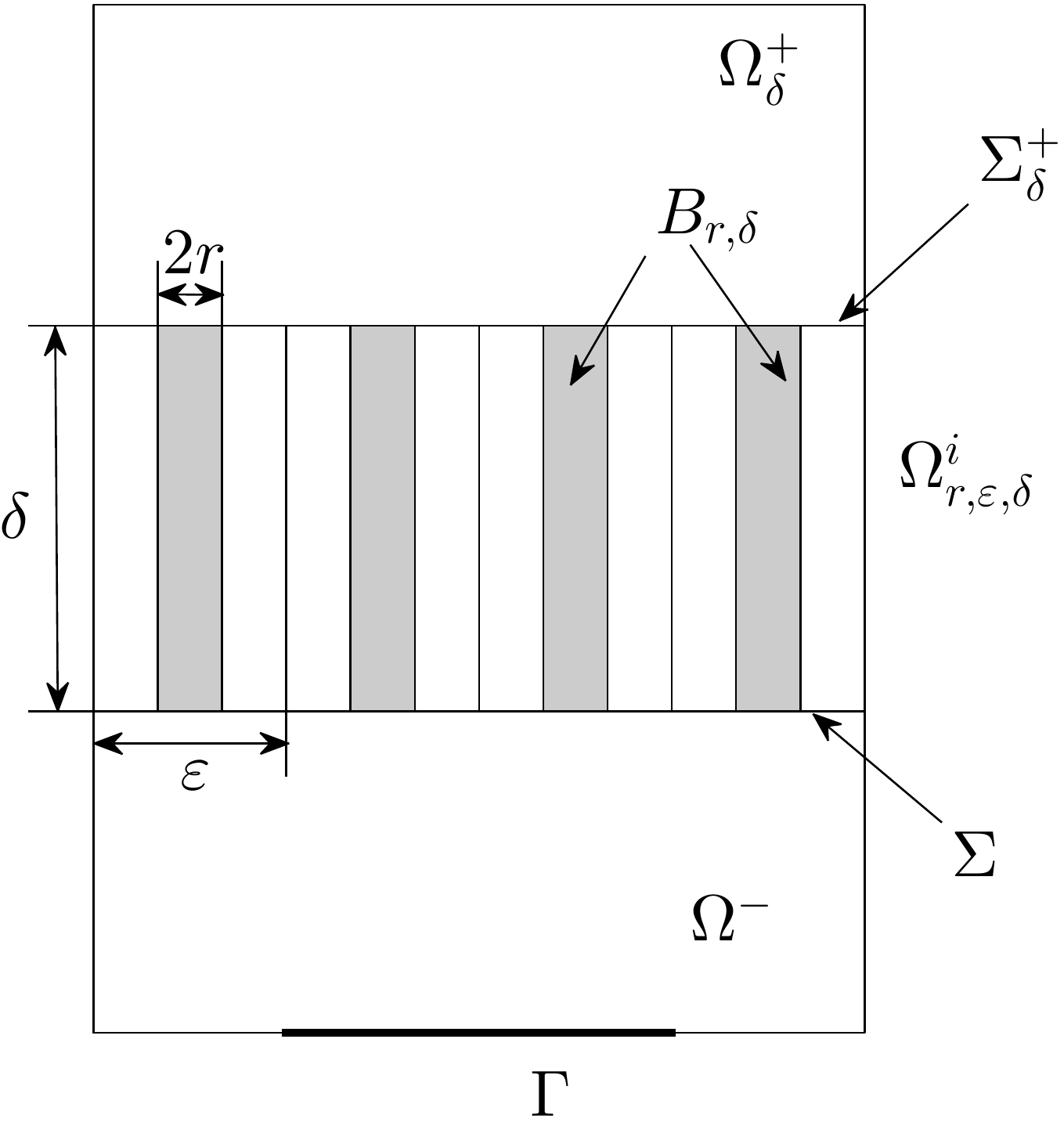}}
\hspace{2cm}
\subfigure[The limit problem]{\includegraphics[width=0.38\textwidth]{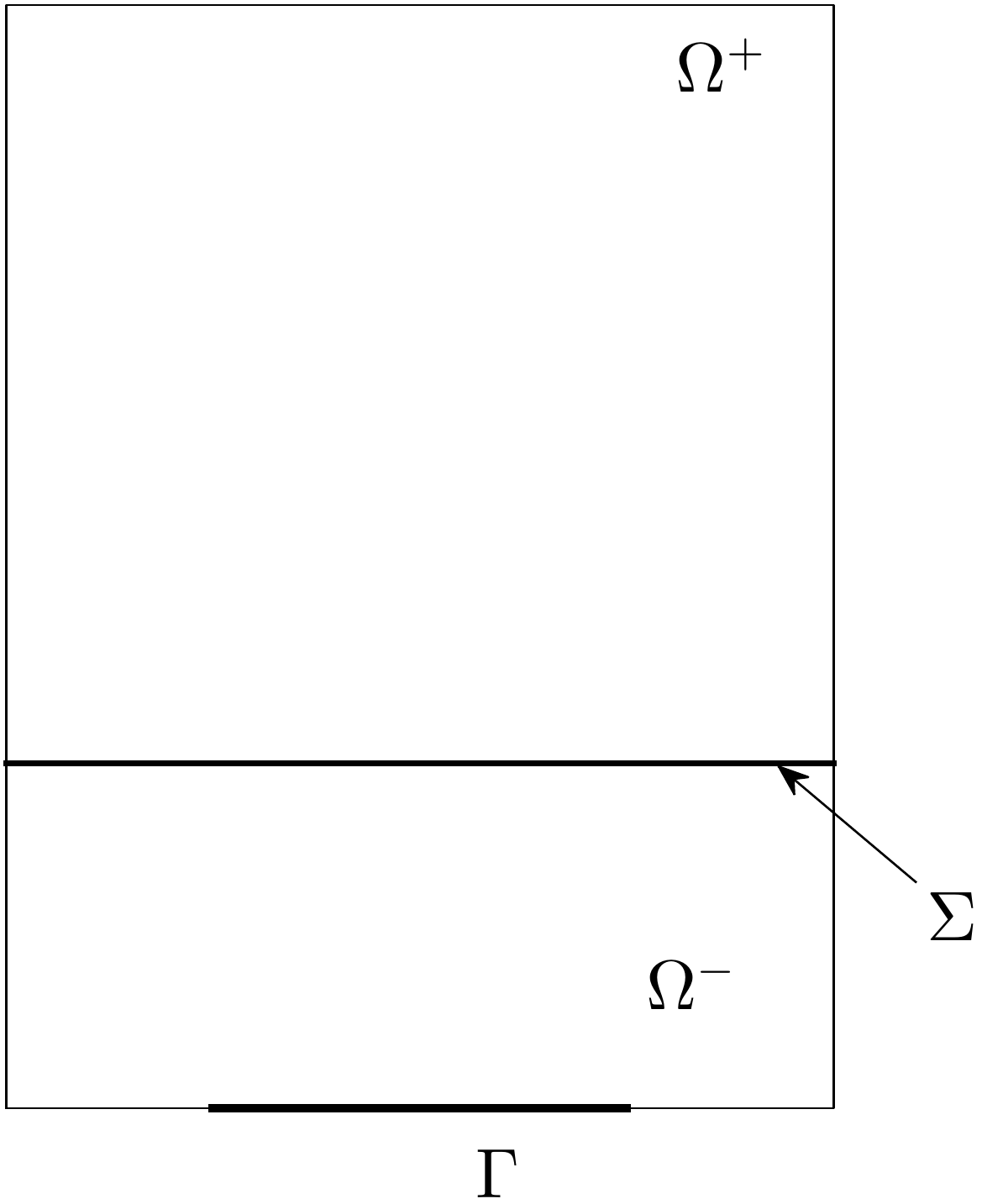}}
\caption{The reference configuration}
\label{limfig}
\end{figure}

The structure is fixed on a part $\Gamma$ with non null measure of the boundary $\partial \Omega^-\setminus \Sigma$.
\medskip

We make the following assumptions:
\begin{equation}
\label{initassump}
r < \frac{\e}{2}, \qquad \frac{r}{\delta} \leq C.
\end{equation}
Here, the first assumption \eqref{initassump}$_1$ is a non penetration condition for the beams while with the second one, we want  to eliminate the case $\displaystyle \frac{\delta}{r} \to 0$ which needs the use of tools for   plates (see \cite{griso}).

\subsection{Strong formulation}\label{SF}
Choose an isotropic material with Lam\'e constants $\lambda^m, \mu^m$ for the beams and another isotropic material with Lam\'e constants $\lambda^b, \mu^b$ for $\O^-$ and $\O^+_\d$. Then we have the following values for the Poisson's coefficient of the material and Young's modulus:
\begin{gather*}
 \nu^m = \frac{\lambda^m}{2(\lambda^m + \mu^m)}, \quad \nu^b = \frac{\lambda^b}{2(\lambda^b + \mu^b)},\\
 E^m = \frac{\mu^m (3 \lambda^m + 2\mu^m)}{\lambda^m + \mu^m}, \quad E^b = \frac{\mu^b (3 \lambda^b + 2\mu^b)}{\lambda^b + \mu^b}.
\end{gather*}

The symmetric deformation field is defined by
$$
(\nabla u)_S = \frac{\nabla u + \nabla^{T} u}{2}.
$$ 

The Cauchy stress tensor in $\Omega_{r, \e, \delta}$ is linked to $(\nabla u_{r, \e, \delta})_S$ through the standard Hooke's law:
\begin{gather*}
 \sigma_{r, \e, \delta} = \left \{
\begin{array}{ll}
 \lambda^b (\mathrm{Tr} \, (\nabla u_{r, \e, \delta})_S) I + 2 \mu^b (\nabla u_{r, \e, \delta})_S &\quad \text{ in } \Omega^- \cup \Omega_{\delta}^+,\\
 \lambda^m (\mathrm{Tr} \, (\nabla u_{r, \e, \delta})_S) I + 2 \mu^m (\nabla u_{r, \e, \delta})_S &\quad \text{ in } \Omega_{r, \e, \delta}^i.
\end{array}
 \right.
\end{gather*}

We consider the standard linear equations of elasticity in $\Omega_{r, \e,\delta}$. The unknown displacement $u_{r, \e,\delta} : \Omega_{r, \e,\delta} \rightarrow \R^3$ satisfies the following problem:
\begin{equation}
\label{strf}
\left\{
\begin{array}{ll}
 \nabla \cdot \sigma_{r, \e, \delta}= -f_{r, \e,\delta} & \text{in } \Omega_{r, \e,\delta}, \\
 u_{r, \e,\delta} = 0 & \text{on } \Gamma,\\
 \sigma_{r, \e, \delta} \cdot \nu = 0 & \text{on } \partial \Omega_{r, \e,\delta} \setminus \Gamma.
 \end{array}
 \right.
\end{equation}

\subsection{Weak formulation}
If $\V_{r,\e,\d}$ denotes the space
\begin{equation*}
\V_{r,\e,\d} = \left\{v \in H^1(\Omega_{r, \e,\delta},\R^3) \;|\; v = 0 \text{ on } \Gamma \right\},
\end{equation*}
the variational formulation of \eqref{strf} is
\begin{equation}
\label{wf}
\left\{
\begin{array}{l}
\text{Find } u_{r, \e,\delta} \in \V_{r,\e,\d},\\[3mm]
\ds \int_{\Omega_{r, \e,\delta}} \sigma_{r, \e, \delta} : (\nabla \varphi)_S dx = \int_{\Omega_{r, \e,\delta}} f_{r, \e,\delta} \cdot \varphi dx, \quad \forall \varphi \in \V_{r,\e,\d}.
\end{array}
\right.
\end{equation}

Throughout the paper and for any $v\in \V_{r,\e,\d}$ we denote by
$$
\begin{aligned}
 \sigma(v) = \lambda(\mathrm{Tr} \, (\nabla v)_S) I + 2 \mu (\nabla v)_S =
 \left \{
\begin{array}{ll}
 \lambda^b (\mathrm{Tr} \, (\nabla v)_S) I + 2 \mu^b (\nabla v)_S &\quad \text{ in } \Omega^- \cup \Omega_{\delta}^+,\\[2mm]
 \lambda^m (\mathrm{Tr} \, (\nabla v)_S) I + 2 \mu^m (\nabla v)_S &\quad \text{ in } \Omega_{r, \e, \delta}^i.
\end{array}
 \right.
\end{aligned}
$$
and 
$${\cal E}(v)=\int_{\Omega_{r, \e,\delta}}\sigma(v): (\nabla v)_S\, dx$$ 
the total elastic energy of the displacement $v$. Indeed choosing $v=u_{r, \e,\delta}$  in \eqref{wf} leads
to the usual energy relation
\begin{equation}\label{(2.9)}
{\cal E}(u_{r, \e,\delta})=\int_{\Omega_{r, \e,\delta}} f_{r, \e,\delta} \cdot u_{r, \e,\delta}\, dx.
\end{equation}

We equip the space $\V_{r, \e, \delta}$ with the following norm:
\begin{equation*}
\| u \|_V = \| (\nabla u)_S \|_{L^2 (\Omega_{r, \e, \delta})}.
\end{equation*}
It follows from the 3D--Korn inequality for domain $\Omega^-$:
\begin{equation}\label{Eq19}
\| u \|_{H^1 (\Omega^-)} \leq C \| (\nabla u)_S \|_{L^2 (\Omega^-)}.
\end{equation}

\section{Decomposition of the displacements in $\Omega_{r, \e,\delta}^i$}
\subsection{Displacement of a single beam.  Preliminary estimates}
To obtain a priori estimates on $u_{r, \e,\delta}$ and $(\nabla u_{r, \e,\delta})_S$ we will need Korn's inequalities for this type of domain. However, for a multi-structure like this, it is not convenient to estimate the constant in a Korn's type inequality, because the order of each component of the displacement field may be very different. To overcome this difficulty, we will use a decomposition for the displacements of beams. A displacement of the beam $B_{r, d}$  is decomposed as the sum of three fields,  the first one  stands for the displacement of the center line, the second stands for the rotations of the cross sections  and the last one is the warping, it takes into account the deformations of the cross sections.

We recall the definition of the elementary displacement from \cite{griso}.
\begin{definition}
\label{eldispl}
The elementary displacement $U_e$, associated to $u \in L^1 (B_{r, d} , \R^3)$, is given by
\begin{equation}
\label{d1}
U_{e} (x_1, x_2, x_3) = \mathcal U(x_3) + \mathcal R(x_3) \wedge (x_1 e_1 + x_2 e_2), \quad \hbox{for a.e. }\; x=(x_1, x_2,x_3) \in B_{r,d},
\end{equation}
where 
\begin{equation}
\label{d3}
\left\{
\begin{aligned}
&\mathcal U = \frac{1}{\pi r^2} \int_{D_r} u(x_1, x_2, \cdot)\, dx_1 dx_2,\\
&\mathcal R_3 = \frac{2}{\pi r^4} \int_{D_r} \left(x_1u_2(x_1, x_2, \cdot)-x_2 u_1(x_1, x_2, \cdot)  \right)  \,dx_1 dx_2,\\
&\mathcal R _{\alpha} = \frac{4(-1)^{3-\alpha}}{\pi r^4} \int_{D_r} x_{3-\alpha}u_3(x_1, x_2, \cdot) \, dx_1 dx_2,\quad \alpha=1,2.
\end{aligned}
\right.
\end{equation}
We write 
\begin{equation}
\label{d2}
\bar u = u - U_e.
\end{equation}
The displacement $\bar u$ is the warping. Note that
\begin{equation}\label{prop-warp}
\begin{aligned}
&\int_{D_r} \bar u(x_1, x_2, \cdot)\,dx_1 dx_2=0,\\
&\int_{D_r} \left(x_1\bar u_2(x_1, x_2, \cdot)-x_2 \bar u_1(x_1, x_2, \cdot)  \right)\,dx_1 dx_2=0,\\
&\int_{D_r} x_1\bar u_3(x_1, x_2, \cdot) \; dx_1 dx_2=\int_{D_r} x_2\bar u_3(x_1, x_2, \cdot)\,dx_1 dx_2=0.
\end{aligned}
\end{equation}
\end{definition}

The following theorem is proved in \cite{griso}.
\begin{theorem}
\label{th1}
Let $u$ be in $H^1(B_{r, d} ; \R^3)$  and $u = U_e + \bar u$ the decomposition of $u$ given by  \eqref{d1}--\eqref{d2}. There exists a constant $C$ independent of $d$ and  $r$ such that the following estimates hold:
\begin{gather}
\label{eq6}
\begin{aligned}
&\| \bar u \|_{L^2 (B_{r, d})} \leq C r \| (\nabla u)_S \|_{L^2 (B_{r, d})}, \qquad \| \nabla \bar u \|_{L^2 (B_{r, d})} \leq C \| (\nabla u)_S \|_{L^2 (B_{r, d})},\\
&\left \| \frac{d\mathcal R}{dx_3} \right\|_{L^2(0, d)} \leq \frac{C}{r^2} \| (\nabla u)_S\|_{L^2(B_{r, d})},\\
&\left \| \frac{d \mathcal U}{dx_3} - \mathcal R \wedge e_3 \right\|_{L^2(0, d)} \leq \frac{C}{r} \| (\nabla u)_S\|_{L^2(B_{r, d})}.
\end{aligned}
\end{gather}
\end{theorem}
We set 
$$  Y_{\e} = \e Y, \quad V_{\e} = Y_{\e} \times (-\e, 0),\quad B'_{r,\e}=D_r\times (-\e,0),\quad V'_{r,\e,\d}=V_{\e}\cup D_r\times (-\e,\d).$$

\begin{lemma}\label{lem1}
Let $u$ be in $H^1(V'_{r,\e,\d} , \R^3)$ and $u = U_e + \bar u$ the decomposition of the restriction of $u$ to the rod $B'_{r,\e}$ given by  \eqref{d1}--\eqref{d2}. There exists a constant $C$ independent of $\d$, $\e$ and  $r$ such that the following estimates hold:
\begin{gather}
\label{Eq25}
\begin{aligned}
&| \mathcal R (0)|^2\leq \frac{C}{r^3} \| \nabla u \|^2_{L^2 (V_{\e})},\\
&\| \mathcal R \|^2_{L^2 (0, \delta)} \leq C \frac{\delta}{r^3} \| \nabla u \|^2_{L^2 (V_{\e})} + C \frac{\delta^2}{r^4} \| (\nabla u)_S \|^2_{L^2 (B_{r, \delta})},\\
&\left\| \frac{d \mathcal U_{\alpha}}{dx_3} \right\|_{L^2 (0, \delta)}^2 \leq C \frac{\delta}{r^3} \| \nabla u \|^2_{L^2 (V_{\e})} + C \frac{\delta^2}{r^4} \| (\nabla u)_S \|^2_{L^2 (B_{r, \delta})},\\
&\| \mathcal U_3 - \mathcal U_3 (0)\|^2_{L^2(0, \delta)} \leq C \frac{\delta^2}{r^2} \| (\nabla u)_S \|^2_{L^2(B_{r, \delta})},\\
&\| \mathcal U_{\alpha} - \mathcal U_{\alpha} (0)\|^2_{L^2 (0, \delta)} \leq C \frac{\delta^3}{r^3} \| \nabla u \|^2_{L^2 (V_{\e})} + C \frac{\delta^4}{r^4} \| (\nabla u)_S \|^2_{L^2 (B_{r, \delta})},\\
&\| u_{\alpha}(\cdot, \cdot, 0) - \mathcal U_{\alpha}(0) \|^2_{L^2 (Y_{\e})} \leq C \e \| \nabla u \|^2_{L^2 (V_{\e})} + C \frac{\e^2}{r} \| (\nabla u)_S \|^2_{L^2 (V_{\e})},\\
&\| u_{3}(\cdot, \cdot, 0) - \mathcal U_{3}(0) \|^2_{L^2 (Y_{\e})} \leq C \e \| \nabla u_3 \|^2_{L^2 (V_{\e})} + C \frac{\e^2}{r} \| (\nabla u)_S \|^2_{L^2 (V_{\e})}.
\end{aligned}
\end{gather}
\end{lemma}
\begin{proof} Applying the 2D-Poincar\'e-Wirtinger's inequality we obtain the following estimate:
\begin{equation}
\label{eq26}
\| u - \mathcal U \|_{L^2 (B'_{r,\e})} \leq Cr \| \nabla u \|_{L^2 (B'_{r,\e})}
\end{equation}
The constant does not depend on $r$ and $\e$.
\medskip

\noindent{\it Step 1. } Estimate of $\mathcal R(0)$.
\smallskip

 Recalling the definition of $\mathcal R$ from \eqref{d3} and since $\displaystyle \int_{D_r} x_1 dx_1 dx_2 = \int_{D_r}  x_2 dx_1 dx_2 =0$, we can write 
\begin{equation*}
\forall x_3\in [-\e,0], \qquad \mathcal R_1 (x_3) = \frac{4}{\pi r^4} \int_{D_r} x_2\big(u_3(x) - \mathcal U_3(x_3)\big) dx_1 dx_2.
\end{equation*}
By Cauchy's inequality
\begin{equation*}
\begin{aligned}
\forall x_3\in [-\e,0], \qquad |\mathcal R_1(x_3)|^2 &\leq \frac{16}{\pi^2r^8} \int_{D_r} x^2 dx_1dx_2 \times \int_{D_r} (u_3(x) - \mathcal U_3(x_3))^2 dx_1dx_2  \\
&\leq \frac{C}{r^4} \int_{D_{r}} (u_3(x) - \mathcal U_3(x_3))^2 dx_1 dx_2.
\end{aligned}
\end{equation*}
Integrating with respect to $x_3$  gives
\begin{equation*}
\int_{-\e}^0 |\mathcal R_1(x_3)|^2 dx_3 \leq \frac{C}{r^4} \int_{B_{r,\e}} (u(x) - \mathcal U(x_3))^2 dx.
\end{equation*}
Using \eqref{eq26} we can write
\begin{equation}\label{R1L2}
\| \mathcal R_1 \|_{L^2 (-\e,0)} \leq \frac{C}{r} \| \nabla u \|_{L^2 (B'_{r,\e})}.
\end{equation}
The derivative of $\mathcal R_1$ is equal to $\displaystyle \frac{d \mathcal R_1}{dx_3} (x_3) = \frac{4}{\pi r^4} \int_{D_r} x_2  \frac{\partial u_3(x)}{\partial x_3} dx_1 dx_2$ for a.e. $x_3\in (-\e,0)$. Then proceeding as above we obtain for a.e. $x_3\in (-\e,0)$
$$
\left| \frac{d \mathcal R_1}{dx_3} (x_3) \right|^2 \leq \frac{C}{r^4} \int_{D_r} \left|\frac{\partial u_3(x)}{\partial x_3} \right|^2 dx_1 dx_2.
$$
Hence
\begin{equation}\label{dR1L2}
\left\| \frac{d \mathcal R_1}{dx_3} \right\|_{L^2(-\e,0)} \leq \frac{C}{r^2} \left\|\frac{\partial u_3}{\partial x_3} \right\|_{L^2(B'_{r,\e})}\leq \frac{C}{r^2} \| \nabla u \|_{L^2 (B'_{r,\e})}.
\end{equation}
We recall the following classical estimates for $\phi\in H^1(-a,0)$ ($a>0$)
\begin{equation}\label{REC}
\begin{aligned}
&|\phi(0)|^2\le \frac{2}{a}\|\phi\|^2_{L^2(-a,0)}+\frac{a}{2}\|\phi'\|^2_{L^2(-a,0)},\\
&\|\phi\|^2_{L^2(-a,0)}\le {2 a}|\phi(0)|^2+{a^ 2}\|\phi'\|^2_{L^2(-a,0)}.
\end{aligned}
\end{equation}
Due to \eqref{R1L2}-\eqref{dR1L2}, \eqref{REC}$_1$ with $a=r$ and since $\e > r$ that gives  for $\mathcal R_1(0)$
\begin{gather*}
| \mathcal R_1 (0)|^2 \leq \frac{C}{r^3} \| \nabla u \|^2_{L^2 (B'_{r,\e})}. 
\end{gather*}
The estimates for $\mathcal R_2(0)$, $\mathcal R_3(0)$ are obtained in the same way.  Hence we get \eqref{Eq25}$_1$.
\medskip

\noindent{\it Step 2. } Estimate of $\| \mathcal R \|_{L^2 (0, \delta)}$.
\smallskip

The Poincar\'e's inequality leads to
$$\| \mathcal R - \mathcal R(0)\|_{L^2 (0, \delta)} \leq \delta \left \| \frac{d \mathcal R}{dx_3} \right\|_{L^2(0, \delta)}.$$
From \eqref{eq6}$_3$, \eqref{REC}$_2$ and \eqref{Eq25}$_1$ we get
\begin{equation}
\label{estr}
\| \mathcal R \|^2_{L^2 (0, \delta)} \leq 2\delta | \mathcal R(0) |^2 + \delta^2 \left \| \frac{d \mathcal R}{dx_3} \right\|^2_{L^2(0, \delta)} \leq C \frac{\delta}{r^3} \| \nabla u \|^2_{L^2 (B'_{r,\e})} + C \frac{\delta^2}{r^4} \| (\nabla u)_S \|^2_{L^2 (B_{r, \delta})}.
\end{equation}
Hence \eqref{Eq25}$_2$ is proved.
\medskip

\noindent{\it Step 3. } Estimate of $\mathcal U-\mathcal U(0)$. 
\smallskip

Applying inequality \eqref{eq6}$_4$ from Theorem \ref{th1} the following estimates on $\mathcal U$ hold:
\begin{gather}
\label{estu0}
\begin{aligned}
\left\| \frac{d \mathcal U_3}{dx_3} \right\|_{L^2 (0, \delta)} & \leq \frac{C}{r} \| (\nabla u)_S \|_{L^2(B_{r, \delta})},\\
\left\| \frac{d \mathcal U_{\alpha}}{dx_3} \right\|_{L^2 (0, \delta)} & \leq \| \mathcal R \|_{L^2 (0, \delta)} + \frac{C}{r} \| (\nabla u)_S \|_{L^2(B_{r, \delta})}.
\end{aligned}
 \end{gather}
Combining \eqref{estu0}$_2$ with \eqref{estr} gives
\begin{equation*}
\left\| \frac{d \mathcal U_{\alpha}}{dx_3} \right\|_{L^2 (0, \delta)}^2 \leq C \frac{\delta}{r^3} \| \nabla u \|^2_{L^2 (B'_{r,\e})} + C \frac{\delta^2}{r^4} \| (\nabla u)_S \|^2_{L^2 (B_{r, \delta})} + \frac{C}{r^2} \| (\nabla u)_S \|^2_{L^2(B_{r, \delta})}.
 \end{equation*}
Taking into account the assumption \eqref{initassump}$_2$, we obtain \eqref{Eq25}$_3$.
Then by \eqref{Eq25}$_3$, \eqref{estu0}$_1$ and the Poincar\'e's inequality \eqref{Eq25}$_4$, \eqref{Eq25}$_5$ follow.
\medskip

\noindent{\it Step 4.} We prove the estimates \eqref{Eq25}$_6$-\eqref{Eq25}$_7$.
\smallskip

By Korn inequality there exists rigid displacement ${\bf r}$
\begin{gather*}
\begin{aligned}
&{\bf r}(x) = {\bf a} + {\bf b} \wedge \left( x + \frac{\e}{2} e_3 \right),\\
&{{\bf a} = \frac{1}{\e^3}\int_{V_{\e}} u(x) \, dx},\\
&{\bf b} = \frac{6}{\e^5} \int_{V_{\e}} \left( x + \frac{\e}{2} e_3 \right) \wedge u(x) \, dx.
\end{aligned}
\end{gather*}
such that
\begin{gather}
\label{korncons}
\begin{aligned}
\| u - {\bf r} \|_{L^2 (V_{\e})} \leq C \e \| (\nabla u)_S \|_{L^2 (V_{\e})},\\
\| \nabla(u - {\bf r}) \|_{L^2 (V_{\e})} \leq C \| (\nabla u)_S \|_{L^2 (V_{\e})}.
\end{aligned}
\end{gather}
Besides by Poincar\'e-Wirtinger inequality we have
\begin{equation}
\label{pwa}
\| u_i - {\bf a}_i \|_{L^2 (V_{\e})} \leq C \e \| \nabla u_i \|_{L^2 (V_{\e})},\qquad i=1,2,3.
\end{equation}

The Sobolev embedding theorems give ($V=Y\times (-1,0)$)
\begin{equation*}
\| \varphi \|_{L^4 (Y)} \leq C \| \varphi \|_{H^{1/2} (Y)} \leq C \left( \| \varphi \|_{L^2 (V)} + \| \nabla \varphi \|_{L^2 (V)} \right), \quad \forall \varphi \in H^1 (V).
\end{equation*}
By a change of variables we obtain
\begin{equation*}
\| \varphi \|_{L^4 (Y_{\e})} \leq C \Big( \frac{1}{\e}\| \varphi \|_{L^2 (V_{\e})} + \e \| \nabla \varphi \|_{L^2 (V_{\e})} \Big), \quad \forall \varphi \in H^1 (V_{\e}).
\end{equation*}
Therefore,  \eqref{korncons} and the above inequality lead to
\begin{equation}
\label{inequr}
\| u - {\bf r} \|_{L^4 (Y_{\e})} \leq C \| (\nabla u)_S \|_{L^2 (V_{\e})}.
\end{equation}
From the identity
\begin{equation*}
\frac{1}{\pi r^2} \int_{D_r} (u(x', 0) - {\bf r} (x', 0)) dx' = \mathcal U (0) - {\bf a} - {\bf b} \wedge \frac{\e}{2} e_3,
\end{equation*}
estimate \eqref{inequr} and the H\"older inequality we get
\begin{equation}
\label{inequ}
\left|\mathcal U (0) - {\bf a} - {\bf b} \wedge \frac{\e}{2} e_3\right| \leq \frac{1}{\pi r^2} \left(\int_{D_r} 1^{4/3} dx' \right)^{3/4} \left( \int_{D_r} |u(x', 0) - {\bf r} (x', 0)|^4 dx' \right)^{1/4} \leq \frac{C }{r^{1/2}} \| (\nabla u)_S \|_{L^2 (V_{\e})}.
\end{equation}
As a first consequence, we obtain
\begin{equation}
\label{inequ11}
\left|\mathcal U_3(0) - {\bf a}_3\right| \leq  \frac{C }{r^{1/2}} \| (\nabla u)_S \|_{L^2 (V_{\e})}.
\end{equation}
From the Cauchy-Schwarz inequality and taking into account \eqref{pwa}, we derive
\begin{multline}
\label{ineqcs}
|{\bf b}| \leq \frac{C}{\e^5} \left( \int_{V_{\e}} \left| x + \frac{\e}{2} e_3 \right|^2 dx \right)^{1/2}  \left( \int_{V_{\e}} | u(x) - {\bf a} |^2 dx \right)^{1/2}  \\
\leq \frac{C}{\e^5} \cdot \e \cdot \e^{3/2} \| u - {\bf a} \|_{L^2 (V_{\e})} \leq \frac{C}{\e^{5/2}} \e \| \nabla u \|_{L^2 (V_{\e})} \leq \frac{C}{\e^{3/2}} \| \nabla u \|_{L^2 (V_{\e})}.
\end{multline}
Using \eqref{inequ} and \eqref{ineqcs} we have
\begin{equation}
\label{inequ2}
\left| \mathcal U(0) - {\bf a} \right| \leq \left| \mathcal U (0) - {\bf a} - {\bf b} \wedge \frac{\e}{2} e_3 \right| + \left| {\bf b} \wedge \frac{\e}{2} e_3 \right| \leq \frac{C }{r^{1/2}} \| (\nabla u)_S \|_{L^2 (V_{\e})} + \frac{C}{\e^{1/2}} \| \nabla u \|_{L^2 (V_{\e})}.
\end{equation}
Estimates \eqref{REC} and \eqref{pwa} yield
\begin{equation}
\label{utr}
\| u_i(\cdot, \cdot, 0) - {\bf a}_i \|^2_{L^2(Y_{\e})} \leq C \e \| \nabla u_i \|^2_{L^2 (V_{\e})},\qquad i=1,2,3.
\end{equation}
Combining \eqref{inequ2}, \eqref{utr} gives
\begin{equation*}
\begin{aligned}
\| u_{\alpha}(\cdot, \cdot, 0) - \mathcal U_{\alpha}(0) \|^2_{L^2 (Y_{\e})} &\leq C (\| u_{\alpha} (\cdot, \cdot, 0) - {\bf a}_{\alpha}\|^2_{L^2 (Y_{\e})} + \| \mathcal U_{\alpha}(0) - {\bf a}_{\alpha} \|^2_{L^2 (Y_{\e})}) \\
&\leq C \e \| \nabla u \|^2_{L^2 (V_{\e})} + C \frac{\e^2}{r} \| (\nabla u)_S \|^2_{L^2 (V_{\e})} + C\e \| \nabla u_{\alpha} \|^2_{L^2 (V_{\e})} \\
&\leq C \e \| \nabla u\|^2_{L^2 (V_{\e})} + C \frac{\e^2}{r} \| (\nabla u)_S \|^2_{L^2 (V_{\e})}, 
\end{aligned}
\end{equation*}
and from \eqref{inequ11} and again \eqref{utr} we obtain
\begin{equation*}
\| u_{3}(\cdot, \cdot, 0) - \mathcal U_{3}(0) \|^2_{L^2 (Y_{\e})} \leq C \e \| \nabla u_{3} \|^2_{L^2 (V_{\e})} + C \frac{\e^2}{r} \| (\nabla u)_S \|^2_{L^2 (V_{\e})}, 
\end{equation*}
Hence we get \eqref{Eq25}$_6$-\eqref{Eq25}$_7$. 
\end{proof}

\section{A priori estimates} 

\noindent In this section all the constants do not depend on $\e, \delta$ and $r$. We denote $x' = (x_1, x_2)$ the running point of $\R^2$.

\subsection{Decomposition of the displacements in  $\Omega_{r, \e,\delta}^i$ }

We decompose the displacement $u\in\V_{r,\e,\d}$ in each beam $\e{\bf i}+B_{r,\delta}$, ${\bf i} \in \widehat{\Xi}_\e \times \{ 0 \}$ as in the Definition \ref{eldispl}.  The components of the elementary displacement are denoted ${\cal U}_\xi$, ${\cal R}_\xi$, where $\xi =\displaystyle \left[ \frac{x'}{\e} \right]_{Y}$.

Now we define the fields $ \widetilde{\mathcal U}$, $ \widetilde{\mathcal R}$ and $\widetilde{\bar{u}}$ for a.e. $x \in B_{r, \delta}, s \in \omega$ by
\begin{equation*}
\widetilde{\mathcal U}(s_1, s_2, x_3) = \left\{\begin{aligned}
 & \mathcal U_{\xi} (x_3), &\hbox{if}\;\; \xi = \left[ \frac{s}{\e} \right]\in \widehat{\Xi}_\e\\
 & 0, &\hbox{if}\;\; \xi\not \in \widehat{\Xi}_\e\\
 \end{aligned}\right.
,\qquad \widetilde{\mathcal R}(s_1, s_2, x_3) = \left\{\begin{aligned}
 & \mathcal R_{\xi} (x_3), &\hbox{if}\;\; \xi = \left[ \frac{s}{\e} \right]\in \widehat{\Xi}_\e\\
 & 0, &\hbox{if}\;\; \xi\not \in \widehat{\Xi}_\e
 \end{aligned}\right.
 ,
\end{equation*}
\begin{equation*}
\widetilde{\bar{u}}(s_1, s_2, x) = \left\{\begin{aligned}
 & \bar u_{\xi} (x), &\hbox{if}\;\; \xi = \left[ \frac{s}{\e} \right]\in \widehat{\Xi}_\e\\
 & 0, &\hbox{if}\;\; \xi\not \in \widehat{\Xi}_\e
 \end{aligned}\right.
 .
\end{equation*}
We have
$$ \widetilde{\mathcal U}, \widetilde{\mathcal R} \in L^2(\omega , H^1((0,\delta) , \R^3)), \quad \widetilde{\bar{u}} \in L^2 (\omega, H^1 (B_{r, \delta}, \R^3)).$$
Moreover,
\begin{gather*}
\| \widetilde {\mathcal U} \|^2_{L^2 (\omega \times (0, \delta))} = \e^2 \sum_{\xi \in \widehat{\Xi}_\e} \| \mathcal U_{\xi} \|^2_{L^2 (0, \delta)},\quad
\| \widetilde {\mathcal R} \|^2_{L^2 (\omega \times (0, \delta))} = \e^2 \sum_{\xi \in \widehat{\Xi}_\e} \| \mathcal R_{\xi} \|^2_{L^2 (0, \delta)}, \\
\|\widetilde{\bar{u}}\|^2_{L^2 (\omega \times B_{r, \delta})} = \e^2 \sum_{\xi \in \widehat{\Xi}_\e} \| \bar{u}_{\xi} \|^2_{L^2 (B_{r, \delta})}.
\end{gather*}
As a consequence of the Theorem \ref{th1} and Lemma \ref{lem1} we get
\begin{lemma}\label{lem31}
Let $u$ be in $\V_{r,\e,\d}$. The following estimates hold:
\begin{equation}\label{344}
\begin{aligned}
&\Big\| \frac{\partial\widetilde {\mathcal R}}{\partial x_3}\Big\|_{L^2 (\omega \times (0, \delta))} \leq  C \frac{\e}{r^2} \| u \|_{V},\\
&\Big\|\frac{\partial\widetilde {\mathcal U}}{\partial x_3}- \widetilde {\mathcal R}\land e_3\Big\|_{L^2 (\omega \times (0, \delta))} \leq C  \frac{\e }{r} \| u \|_{V},\\
&\|\nabla_x \widetilde{\bar{u}}\|_{L^2 (\omega \times B_{r, \delta})} \leq C \e \| u \|_{V} ,\\
&\|\widetilde{\bar{u}}\|_{L^2 (\omega \times B_{r, \delta})} \leq C \e r \| u \|_{V} ,\\
&\| \widetilde {\mathcal R} \|_{L^2 (\omega \times (0, \delta))} \leq C \frac{\e\delta}{r^2} \| u \|_{V},\\
&\Big\|\frac{\partial\widetilde {\mathcal U}_\alpha}{\partial x_3}\Big\|_{L^2 (\omega \times (0, \delta))}\leq C \frac{\e\delta}{r^2} \| u \|_{V}.
\end{aligned}
\end{equation}
Moreover,
\begin{equation}\label{Diff}
\begin{aligned}
&\| \widetilde{\mathcal R} (\cdot, \cdot, 0)\|^2_{L^2 (\widehat{\omega}_{\e})} \leq C \frac{\e^2}{r^{3}} \|u \|^2_V,\\
&\| \widetilde{\mathcal R} (\cdot, \cdot, \delta)\|^2_{L^2 (\widehat{\omega}_{\e})} \leq C \frac{\e^2}{r^{3}} \| \nabla u \|^2_{L^2(\Omega^+_{\delta})},\\
&\| \widetilde{\mathcal U}_3-\widetilde{\mathcal U}_3(\cdot,\cdot, 0)\|_{L^2(\omega \times(0, \delta))} \leq C \frac{\delta \e}{r} \| u \|_{V},\\
&\| \widetilde{\mathcal U}_{\alpha}-\widetilde{\mathcal U}_{\alpha}(\cdot,\cdot, 0)\|_{L^2 (\omega \times (0, \delta))} \leq C \frac{\delta^2 \e}{r^2} \| u \|_{V}, \quad \text{where } \alpha = 1,2.
\end{aligned}
\end{equation}
\end{lemma}
\begin{proof}
Estimates \eqref{344}$_1$ -- \eqref{344}$_6$ follow directly from \eqref{Eq19}, \eqref{eq6}$_3$, \eqref{eq6}$_4$  and \eqref{Eq25}$_2$--\eqref{Eq25}$_3$ and estimates \eqref{Diff}$_1$ -- \eqref{Diff}$_4$ are the consequences of the estimates in Lemma \ref{lem1} and \eqref{Eq19}.
\end{proof}

\subsection{Estimates of the interface traces}
\begin{lemma}\label{lem32} There exists a constant $C$ independent of $\e, \delta, r$ such that for any $u \in \V_{r,\e,\d}$
\begin{equation}
\label{estu1}
\| u(\cdot, \cdot, 0) - \widetilde{\mathcal U}(\cdot, \cdot, 0) \|^2_{L^2 (\widehat{\omega}_{\e})} \leq C \frac{\e^2}{r} \| u \|^2_{V},
\end{equation}
\begin{equation}\label{estu2}
\begin{aligned}
\| u_{\alpha}(\cdot, \cdot, \delta) - \widetilde{\mathcal U}_{\alpha}(\cdot, \cdot, \delta) \|^2_{L^2 (\widehat{\omega}_{\e})} \leq C \e \| \nabla u_{\alpha} \|^2_{L^2 (\Omega^+_{\delta})} + C \frac{\e^2}{r} \| u \|^2_{V},\\
\| u_{3}(\cdot, \cdot, \delta) - \widetilde{\mathcal U}_{3}(\cdot, \cdot, \delta) \|^2_{L^2 (\widehat{\omega}_{\e})} \leq C \e \| \nabla u_{3} \|^2_{L^2 (\Omega^+_{\delta})} + C \frac{\e^2}{r} \| u \|^2_{V}.
\end{aligned}
\end{equation}
Moreover,
\begin{gather}
\label{diffu1}
\| u_{\alpha}(\cdot, \cdot, \delta) - u_{\alpha}(\cdot, \cdot, 0) \|^2_{L^2 (\widehat{\omega}_{\e})} \leq C \e \| \nabla u_{\alpha} \|^2_{L^2 (\Omega^+_{\delta})} + C \frac{\e^2 \delta^3}{r^4} \| u \|^2_V,\\
\label{diffu2}
\| u_3(\cdot, \cdot, \delta) - u_3(\cdot, \cdot, 0) \|^2_{L^2 (\widehat{\omega}_{\e})} \leq C \e \| \nabla u_{3} \|^2_{L^2 (\Omega^+_{\delta})} + C \frac{\e^2 \delta}{r^2} \| u \|^2_V.
\end{gather}
\end{lemma}

\begin{proof} Using \eqref{Eq25}$_6$-\eqref{Eq25}$_7$ and then summing over all cells give \eqref{estu1}. 
In the same way the  estimates \eqref{estu2} are derived.
\smallskip

Applying \eqref{344}$_2$ we  write
\begin{equation}
\label{eq4}
\|\widetilde{\mathcal U}_3(\cdot, \cdot, \delta) - \widetilde{\mathcal U}_3 (\cdot, \cdot, 0)\|_{L^2 (\widehat{\omega}_{\e})}^2 \leq \delta \left\| \frac{\partial \widetilde{\mathcal U}_3}{\partial x_3} \right\|^2_{L^2 (\widehat{\omega}_{\e} \times (0, \delta))} \leq C \frac{\e^2 \delta}{r^2} \| u \|^2_V.
\end{equation}
From \eqref{344}$_6$ we have
\begin{equation}
\label{eq5}
\| \widetilde{\mathcal U}_{\alpha} (\cdot, \cdot, \delta) - \widetilde{\mathcal U}_{\alpha}(\cdot, \cdot, 0)\|_{L^2 (\widehat{\omega}_{\e})}^2 \leq  C \frac{\e^2 \delta^3}{r^4} \| u \|^2_V.
\end{equation}
Using \eqref{eq5} and the above estimates we obtain \eqref{diffu1}, \eqref{diffu2}.
\end{proof}

\subsection{Estimates of the displacements in  $\Omega_{\delta}^+$ }
\begin{lemma}
\label{lem_est-1}
There exists a constant $C$ which does not depend on $\varepsilon$, $r$ and $\delta$, such that  for any $u \in \V_{r,\e,\d}$
\begin{gather}
\label{eq8}
\| u_{\alpha} \|_{H^1 (\Omega_{\delta}^+)} \leq C \Big(\frac{\e \delta^{3/2}}{r^2}+1\Big)\| u \|_V,\\
\label{eq80}
\| u_3 \|_{H^1 (\Omega_{\delta}^+)} \leq C\Big(\frac{\e \d^{1/2}}{r} +1\Big)\| u \|_V,
\end{gather}
where $\alpha = 1, 2$.
\end{lemma}
\begin{proof}
From the Korn's inequality and the trace theorem  we derive
\begin{gather}
\label{estset0}
\begin{aligned}
&\| u \|_{L^2 (\Sigma)} & \leq \quad & C \| u \|_{H^1 (\Omega^-)} \leq C \| (\nabla u)_S \|_{L^2 (\Omega^{-})},\\
&\| u \|_{H^1 (\Omega_{\delta}^+)} & \leq \quad & C \big(\|(\nabla u)_S \|_{L^2(\Omega_{\delta}^+)} + \| u\|_{L^2 (\Sigma^+_\delta)}\big).
\end{aligned}
\end{gather}
We know that there exists a rigid displacement ${\bf r}$ 
$$\forall x\in \R^3,\qquad {\bf r}(x)={\bf a}+{\bf b}\land (x-\delta e_3),\qquad {\bf a},\;{\bf b}\in \R^3,$$ such that
\begin{equation}\label{u-r}
\| u - {\bf r} \|_{L^2 (\Sigma_{\delta}^+)} \leq C \|u-{\bf r}\|_{H^1(\Omega_{\delta}^+)}\leq C\| (\nabla u)_S \|_{L^2(\Omega_{\delta}^+)}.
\end{equation}
The constant does not depend on $\delta$. Then, we get
\begin{equation}\label{Trace u-r}
\|(u-{\bf r}) (\cdot, \cdot, \delta)\|_{L^2(\widehat{\omega}_{\e})} \leq \| u - {\bf r} \|_{L^2 (\Sigma_{\delta}^+)} \leq C\| (\nabla u)_S \|_{L^2(\Omega_{\delta}^+)}.
\end{equation}
Using
\begin{equation}
\| u(\cdot, \cdot, 0) \|_{L^2 (\widehat{\omega}_{\e})} \leq \| u \|_{L^2 (\Sigma)}\leq C\| (\nabla u)_S \|_{L^2(\Omega^-)}\leq C\|u \|_{V},
\end{equation}
from \eqref{diffu1}, \eqref{diffu2} we obtain
\begin{gather}
\label{ulest}
\begin{aligned}
\| u_{\alpha}(\cdot, \cdot, \delta) \|_{L^2 (\widehat{\omega}_{\e})} \leq C \e^{1/2} \| \nabla u_{\alpha} \|_{L^2 (\Omega^+_{\delta})} + C \frac{\e \delta^{3/2}}{r^2} \| u \|_V + C\| u \|_V,\\
\| u_3(\cdot, \cdot, \delta) \|_{L^2 (\widehat{\omega}_{\e})} \leq C \e^{1/2} \| \nabla u_{3} \|_{L^2 (\Omega^+_{\delta})} + C \frac{\e \delta^{1/2}}{r} \| u \|_V + C\| u \|_V.
\end{aligned}
\end{gather}
Combining this with \eqref{Trace u-r} gives
\begin{gather}
\label{rlest}
\begin{aligned}
\| {\bf r}_{\alpha} (\cdot, \cdot, \delta) \|_{L^2(\widehat{\omega}_{\e})} \leq C \e^{1/2} \| \nabla u_{\alpha} \|_{L^2 (\Omega^+_{\delta})} + C \frac{\e \delta^{3/2}}{r^2} \| u \|_V + C\| u \|_V,\\
\| {\bf r}_3 (\cdot, \cdot, \delta) \|_{L^2(\widehat{\omega}_{\e})} \leq C \e^{1/2} \| \nabla u_{3} \|_{L^2 (\Omega^+_{\delta})} + C \frac{\e \delta^{1/2}}{r} \| u \|_V + C\| u \|_V.
\end{aligned}
\end{gather}
Therefore,
\begin{gather*}
|{\bf a_1}| + |{\bf a_2}| + |{\bf b_3}|  \leq C \e^{1/2} \| \nabla u_{\alpha} \|_{L^2 (\Omega^+_{\delta})} + C \frac{\e \delta^{3/2}}{r^2} \| u \|_V + C\| u \|_V,\\
|{\bf a_3}| + |{\bf b_1}| + |{\bf b_2}| \leq C \e^{1/2} \| \nabla u_{3} \|_{L^2 (\Omega^+_{\delta})} + C \frac{\e \delta^{1/2}}{r} \| u \|_V + C\| u \|_V.
\end{gather*}
These estimates together with \eqref{u-r} allow to obtain estimates on $u_1, u_2, u_3$. From this we have
\begin{gather}
\label{h1u}
\begin{aligned}
\| u_{\alpha} \|_{H^1 (\Omega_{\delta}^+)} \leq C \e^{1/2} \| \nabla u_{\alpha} \|_{L^2 (\Omega^+_{\delta})} + C \frac{\e \delta^{3/2}}{r^2} \| u \|_V + C\| u \|_V,\\
\| u_3 \|_{H^1 (\Omega_{\delta}^+)} \leq C \e^{1/2} \| \nabla u_{3} \|_{L^2 (\Omega^+_{\delta})} + C \frac{\e \delta^{1/2}}{r} \| u \|_V + C\| u \|_V.
\end{aligned}
\end{gather}
Therefore, for $\e$ small enough the following hold true:
\begin{equation*}
\| \nabla u_\alpha \|_{L^2 (\Omega_{\delta}^+)} \leq C \frac{\e \delta^{3/2}}{r^2} \| u \|_V + C\| u \|_V,\qquad \| \nabla u_3 \|_{L^2 (\Omega_{\delta}^+)} \leq C \frac{\e \delta^{1/2}}{r} \| u \|_V + C\| u \|_V.
\end{equation*}
Inserting that in \eqref{h1u} we derive \eqref{eq8}-\eqref{eq80}.
\end{proof}
As a consequence of Lemma \ref{lem_est-1}, the estimates \eqref{diffu1}, \eqref{diffu2} can be replaced by
\begin{gather}
\label{est_cas1}
\| u_{\alpha}(\cdot, \cdot, \delta) - u_{\alpha}(\cdot, \cdot, 0) \|^2_{L^2 (\widehat{\omega}_{\e})} \leq C \frac{\e^2 \delta^3}{r^4} \| u \|^2_V,\\
\label{est_cas2}
\| u_3(\cdot, \cdot, \delta) - u_3(\cdot, \cdot, 0) \|^2_{L^2 (\widehat{\omega}_{\e})} \leq C  \frac{\e^2 \delta}{r^2} \| u \|^2_V.
\end{gather}


\subsection{Estimates for the set of beams $\Omega_{r, \e, \delta}^i$}
\begin{lemma}
\label{lem_est-2}
There exists a constant $C$ which does not depend on $\varepsilon$, $r$ and $\delta$, such that for any $u \in \V_{r,\e,\d}$
\begin{equation}\label{Eq321}
\begin{aligned}
\|\nabla u\|_{L^2 (\Omega_{r, \e, \delta}^i)}&\le C \frac{\delta}{r} \| u \|_V,\\
\|u_3\|_{L^2 (\Omega_{r, \e, \delta}^i)}&\le C  \frac{r\d^{1/2}}{\e} \| u \|_V,\\
\|u_\alpha\|_{L^2 (\Omega_{r, \e, \delta}^i)}& \leq C \frac{r\d^{1/2}}{\e}\left(1 + \frac{\e\d^{3/2}}{r^2} \right) \| u \|_V
\end{aligned}
\end{equation}
where $\alpha = 1, 2$.
\end{lemma}

\begin{proof} 
From the estimates in Theorem \ref {th1}, \eqref{Eq25}$_2$ and \eqref{Eq25}$_3$ and after summation over all the beams, we get (we make use of the  assumption \eqref{initassump}$_2$)
\begin{equation}\label{grad u}
\|\nabla u\|^2_{L^2 (\Omega_{r, \e, \delta}^i)}\le C \left( \frac{\delta}{r} \| \nabla u \|^2_{L^2 (\Omega^-)} +\frac{\delta^2}{r^2} \| (\nabla u)_S \|^2_{L^2 (\Omega_{r, \e, \delta}^i)}\right) \leq C \frac{\delta^2}{r^2} \| u \|^2_V.
\end{equation}
From  \eqref{estu1} and \eqref{estset0}$_1$, it follows that
\begin{gather*}
\sum_{\xi \in \Xi_{\e}} \e^2 | \mathcal U_{\xi} (0) |^2  = \| \widetilde{\mathcal U} (\cdot, \cdot, 0) \|^2_{L^2 (\widehat{\omega}_{\e})} \leq C \frac{\e^2}{r}\| u \|^2_V + C\| u \|^2_V,\\
\sum_{\xi \in \Xi_{\e}}| \mathcal U_{\xi} (0) |^2 \leq C \left(\frac{1}{r} + \frac{1}{\e^2}\right)\| u \|^2_V.
\end{gather*}
Using \eqref{eq6}$_4$, \eqref{Eq25}$_3$, \eqref{REC},  we obtain
\begin{gather}
\label{Ua3}
\begin{aligned}
\sum_{\xi \in \Xi_{\e}}\| \mathcal U_{\xi, 3} \|^2_{L^2 (0, \delta)} \leq C \left( \frac{\d}{\e^2} + \frac{\delta^2}{r^2} \right) \| u \|^2_V,\\
\sum_{\xi \in \Xi_{\e}}\| \mathcal U_{\xi, \alpha} \|^2_{L^2 (0, \delta)} \leq C \left( \frac{\d}{\e^2} + \frac{\delta^4}{r^4} \right) \| u \|^2_V.
\end{aligned}
\end{gather}
Additionally,
\begin{equation}
\label{wrineq2}
\sum_{\xi \in \Xi_{\e}} \| \bar u_{\xi} \|^2_{L^2 (B_{r, \delta})} \leq C r^2 \| (\nabla u)_S \|^2_{L^2 (\Omega_{r, \e, \delta}^i)} \leq C r^2 \| u \|^2_V.
\end{equation}
Then \eqref{Eq25}$_2$, \eqref{Ua3} and \eqref{wrineq2} give
\begin{gather*}
\sum_{\xi \in \Xi_{\e}} \| u_{\xi, \alpha}\|^2_{L^2 (B_{r, \delta})} \le C \left( \frac{r^2\d}{\e^2} + \frac{\delta^4}{r^2} + \delta^2 + r^2 \right) \| u \|^2_V \leq C \frac{r^2\d}{\e^2} \left(1+ \frac{\e^2\delta^3}{r^4} \right) \| u \|^2_V,\\
\sum_{\xi \in \Xi_{\e}} \| u_{\xi, 3}\|^2_{L^2 (B_{r, \delta})} \le C \left( \frac{r^2\d}{\e^2} + \delta^2 + r^2 \right) \| u \|^2_V \leq C \frac{r^2\d}{\e^2} \left( 1 + \frac{\e^2\delta}{r^2} \right) \| u \|^2_V.
\end{gather*}
From the last inequalities we derive \eqref{Eq321}$_2$ and \eqref{Eq321}$_3$.
\end{proof}

\subsection{The limit cases}\label{limcas}
In view of the conditions \eqref{initassump} and the estimates in Lemma  \ref{lem_est-1}, and in order that the lower and upper parts of our structure match, we must assume that
\begin{equation}\label{assump}
\frac{\varepsilon^2 \delta^3}{r^4}  \hbox{ is uniformly bounded from above}.
\end{equation}
%
%
{\it From now on, the parameters $r$, $\d$ and $\e$ are linked in this way
\begin{itemize}
\item $r=\kappa_0\e^{\eta_0}$, $\eta_0\ge 1$, $\kappa_0>0$, if $\eta_0=1$ then $\kappa_0\in (0,1/2)$ (non penetration condition),
\item $\d=\kappa_1\e^{\eta_1}$, $\kappa_1>0$ and $\eta_1 \leq \eta_0$.
\end{itemize}
The above assumption \eqref{assump} yields
$$2+3\eta_1-4\eta_0\ge 0.$$
 
\noindent Hence, the couple $(\eta_0,\eta_1)$ belongs to the triangle whose vertexes are
\begin{equation}\label{triangle}
(2/3,1),\qquad (1,1),\qquad (2,2).
\end{equation}
The case
$$\lim_{\e\to 0}{\e^2\d^3\over r^4}=0$$ could be very easily analysed. Using the estimates \eqref{est_cas1}-\eqref{est_cas2}, in this case we can prove  that, the limit  displacements on both parts  coincide on the interface; hence  the limit displacement  belongs to $H^1(\omega\times(-L,L) , \R^3)$ and it is the solution of an  elasticity system. 
\smallskip

Therefore, the most interesting cases correspond to $\ds \lim_{\e\to 0}{\e^2\d^3\over r^4}>0$; this is the critical situation 
\begin{itemize}
\item (i) $r=\kappa_0\e$,  $\kappa_0\in (0,1/2)$ and $\d=\kappa_1\e^{2/3}$, $\kappa_1>0$,
\item (ii) $r=\kappa_0\e^{\eta_0}$, $\eta_0\in (1,2)$, $\kappa_0>0$ and $\d=\kappa_1\e^{(4\eta_0-2)/3}$, $\kappa_1>0$.
\end{itemize}
We obtain  the edge of the triangle with the vertexes $(2/3, 1)$ and $(2,2)$. We eliminate the case $\eta_0=\eta_1=2$ to deal with  small beams in the layer.
\bigskip


\noindent For the sake of simplicity, from now on we will use the following notations: }
\begin{itemize}
\item $\Omega_{\e}$ instead of $\Omega_{r, \e, \delta}$,
\item $\Omega_{\varepsilon}^i$ instead of $\Omega_{r, \e, \delta}^i$,
\item $\Omega_{\varepsilon}^+$ instead of $\Omega_{\delta}^+$,
\item $\sigma_{\e}$ instead of $\sigma_{r, \e, \delta}$,
\item $u_{\e}$ instead of $u_{r, \e, \delta}$,
\item $f_{\e}$ instead of $f_{r, \e, \delta}$.
\end{itemize} 

With assumption \eqref{assump} we can rewrite some estimates obtained above. For any $u \in \V_{r,\e,\d}$ we have
\begin{gather}
\label{finest}
\| u \|_{L^2 (\Omega_{\e}^i)} \leq C \frac{r\d^{1/2}}{\e} \| u \|_V,\\
\| u\|_{H^1 (\Omega_{\e}^+)} \leq C \| u \|_V.
\label{fineste}
\end{gather}
The constants do not depend on $\e$, $r$ and $\d$.

\subsection{Force assumptions}
We set
$$ B_1 = D_1 \times (0, 1).$$
To obtain estimates on $u_{\e}$ we test \eqref{wf} with $\varphi = u_{\varepsilon}$. We have
\begin{equation}
\label{aprest}
M_1 \| u_{\e}\|^2_V \leq  \| f_{\varepsilon} \|_{L^2 (\Omega_{\e},\R^3)} \| u_{\e} \|_{L^2(\Omega_{\e},\R^3)}.
\end{equation}

We consider the following assumption on the applied forces:
\begin{equation}
\label{forassump}
f _\e(x) = \left\{
\begin{aligned}
&\ds \frac{\e^2}{r^2\d} \,  F^m \left(\e\left[ \frac{x'}{\e}\right]_{Y}, \frac{\e}{r}\left\{ \frac{x'}{\e}\right\}_{Y}, \frac{x_3}{\delta} \right) \enskip \hbox{for a.e.}\;\; x \in \Omega_{\e}^i,\\[2mm]
&F (x) \hskip 39mm \hbox{for a.e.}\;\; x \in \Omega^- \cup \Omega^+_{\e},
\end{aligned}
\right.
\end{equation}
where $F^m \in {\cal C}^0 (\overline{\omega}, L^2 (B_1, \R^3))$, $F \in L^2 (\omega \times (-L, L), \R^3)$. Then,
\begin{equation*}
\| f _\e\|_{L^2 (\Omega^i_{\e} , \R^3)} \leq \frac{\e}{r\delta^{1/2}}  \| F^m \|_{L^\infty(\omega,L^2 (B_1,\R^3))}.
\end{equation*}
Making use of the estimates  \eqref{Eq19}, \eqref{finest}, \eqref{fineste} together with   inequality \eqref{aprest}  yield 
\begin{equation}\label{Eq327}
\|u_\e\|_V\le C
\end{equation}
The constant does not depend of $r$, $\e$ and $\d$.
\medskip

{\it As mention above, from now on, we only consider the cases (i) and (ii) introduced in Section \ref{limcas}.}
\section{The periodic unfolding operators}
\begin{definition}
For $\varphi$ Lebesgue-measurable function on $\omega\times (0,\d)$, the unfolding operator $\mathcal T_{\e}$ is defined as follows:
\begin{equation*}
\mathcal T_{\e} (\varphi) (s_1, s_2, X_3) = 
\left\{
\begin{array}{ll}
 \varphi (s_1, s_2, \delta X_3), & \text{ for a.e.  }\;\;  (s_1, s_2, X_3) \in \widehat \omega_{\e} \times (0, 1),\\
 0, & \text{ for a.e. } \;\; (s_1, s_2, X_3) \in \Lambda_{\varepsilon} \times (0, 1).
\end{array}
\right.
\end{equation*}
\end{definition}

\begin{definition}
For $\varphi$ Lebesgue-measurable function on $\omega\times B_{r,\d}$, the unfolding operator $\mathcal T'_{\e}$ is defined as follows:
\begin{equation*}
\mathcal T^{'}_{\e} (\varphi) (s_1, s_2, X_1, X_2, X_3) = 
\left\{
\begin{array}{ll}
 \varphi (s_1, s_2, rX_1, rX_2, \delta X_3), & \text{ for a.e.  }\;\;  (s_1, s_2, X_1, X_2, X_3) \in \widehat \omega_{\e} \times B_1,\\
 0, & \text{ for a.e.  }\;\;  (s_1, s_2, X_1, X_2, X_3) \in \Lambda_{\varepsilon} \times B_1.
\end{array}
\right.
\end{equation*}
\end{definition}
Observe that if $\varphi$ is a Lebesgue-measurable function on $\omega\times (0,\d)$ then $\mathcal T_{\e}(\varphi)=\mathcal T_{\e}^{'}(\varphi)$.
\begin{lemma}
\label{lpr}
(Properties of the operators $\mathcal T_{\e}$, $\mathcal T_{\e}^{'}$)
 \begin{enumerate}
  \item $\forall v,w \in L^2(\omega \times (0, \delta))$ $$\mathcal T_{\e}(vw) = \mathcal T_{\e}(v) \mathcal T_{\e}(w),$$\\
  $\forall v,w \in L^2(\omega \times B_{r, \delta})$ $$\mathcal T_{\e}^{'}(vw) = \mathcal T_{\e}^{'}(v) \mathcal T_{\e}^{'}(w).$$
  
  \item $\forall u \in L^1(\omega\times (0, \delta))$  
  $$\delta \int_{\omega \times (0, 1)} \mathcal T_{\e}(u) \, ds_1\,ds_2\,dX_3 = \int_{\widehat{\omega}_{\e} \times (0, \delta)} u\,ds_1\,ds_2\,dx_3,$$
  $\forall u \in L^1(\omega\times B_{r, \delta})$  
  $$r^2 \delta \int_{\omega \times B_1} \mathcal T_{\e}^{'}(u)\,ds_1\,ds_2\,dX_1\,dX_2\,dX_3 = \int_{\widehat{\omega}_{\e} \times B_{r, \delta}} u\,ds_1\,ds_2\,dx_1\,dx_2\,dx_3.$$
  
  \item $\forall u \in L^2(\omega \times (0, \delta))$  
  $$\| \mathcal T_{\e}(u) \|_{L^2 (\omega \times (0, 1))} \leq \frac{1}{\sqrt \delta} \| u\|_{L^2(\omega \times (0, \delta))},$$
  
  $\forall u \in L^2(\omega \times B_{r, \delta})$  
  $$\| \mathcal T_{\e}^{'}(u) \|_{L^2 (\omega \times B_1)} \leq \frac{1}{r \sqrt \delta} \| u\|_{L^2(\omega \times B_{r, \delta})}.$$
  
  \item Let $u$ be in $L^2(\omega, H^1(0, \delta))$, a.e.  in $\omega \times (0, 1)$ we have
  $$\delta \mathcal T_{\e}(\nabla_{x_3} u) = \nabla_{X_3} \mathcal T_{\e} (u).$$
  Let $u$ be in $L^2(\omega, H^1(B_{r, \delta}))$., a.e.  in $\omega \times B_1$ we have
  $$r \mathcal T_{\e}^{'}(\nabla_{x_{\alpha}} u) = \nabla_{X_{\alpha}} \mathcal T_{\e}^{'} (u), \quad \delta \mathcal T_{\e}^{'}(\nabla_{x_3} u) = \nabla_{X_3} \mathcal T_{\e}^{'} (u), \text{ where } \alpha = 1,2.$$
  \end{enumerate}
\end{lemma}

\begin{proof}
Properties 1-3 are obtained similarly as in the proof of Lemma 5.1 of \cite{griso2}.

\noindent Property 4 is the direct consequence chain rule formulae:
\begin{gather*}
\frac{\partial(\mathcal T_{\e}^{'} (u))}{\partial X_{\alpha}} = r \mathcal T_{\e}^{'}\left( \frac{\partial u}{\partial x_{\alpha}} \right), \quad \alpha = 1, 2,\\
\frac{\partial(\mathcal T_{\e} (u))}{\partial X_3} = \delta \mathcal T_{\e}\left( \frac{\partial u}{\partial x_3} \right), \quad \frac{\partial(\mathcal T_{\e}^{'} (u))}{\partial X_3} = \delta \mathcal T_{\e}^{'}\left( \frac{\partial u}{\partial x_3} \right).
\end{gather*}
\vskip-9mm
\end{proof}

\subsection{The limit fields (Cases (i) and (ii))}

{\it From now on, $(u_{\varepsilon})_{\alpha}$ will be denoted as $u_{\e,\alpha}$; the same notation will be used for the fields with values in $\R^2$ or $\R^3$.}
\medskip

From Lemmas \ref{lem31} and \ref{lpr} we obtain the following result.
\begin{lemma}
\label{unfest}
There exists a constant $C$ independent of $\e$, $\d$ and $r$ such that
\begin{gather}
\| \mathcal T_{\e}(\widetilde{\mathcal U}_{\e}) \|_{L^2(\omega ,H^1(0, 1))} \leq C  ,\\
\| \mathcal T_{\e}(\widetilde{\mathcal U}_{\e,3}) - \widetilde{\mathcal U}_{\e,3 }(\cdot, \cdot, 0)\|_{L^2(\omega , H^1(0, 1))} \leq C \frac{r}{\delta}  ,\\
\| \mathcal T_{\e}(\widetilde{\mathcal R}_{\e}) \|_{L^2(\omega, H^1 (0, 1))} \leq \frac{C}{\delta}  ,\\
\Big\| \frac{\partial \mathcal T_{\e}(\widetilde {\mathcal U}_{\e})}{\partial X_3}- \delta\mathcal T_{\e}(\widetilde {\mathcal R}_{\e})\land e_3\Big\|_{L^2 (\omega \times (0, 1))} \leq C \frac{r}{\delta}  ,\\
\| \mathcal T_{\e}^{'}(\widetilde{\bar u}_{\e}) \|_{L^2(\omega \times(0,1), H^1( D_1))} \leq C \frac{r^2}{\delta^2},\\
\left\|\frac{\partial \mathcal T_{\e}^{'}(\widetilde{\bar{u}}_{\e})}{\partial X_3}\right\|_{L^2 (\omega \times B_1)} \leq C \frac{r}{\delta}.
\end{gather}
\end{lemma}


Further we extend function $u_{\e}$ defined on the domain $\Omega_{\e}^+$ by reflection to the domain $\omega \times (\delta, L + \delta)$. The new function is  denoted $u_{\e}$ as before.
\begin{proposition}\label{prop41}
\label{limits_case1} There exist a subsequence of $\{\e\}$, still denoted by $\{ \e \}$, and $u^{\pm} \in H^1(\Omega^{\pm},\R^3)$ with $u^- = 0$ on $\Gamma$, $\widetilde{\mathcal R} \in L^2 (\omega, H^1_0 ((0, 1),\R^3))$, $\widetilde{\mathcal U}_\alpha \in L^2 (\omega, H^2(0, 1))$, $\widetilde{\mathcal U}_3,\;\;\widetilde{\mathcal U}'_3 \in L^2 (\omega, H^1(0, 1))$, $\widetilde{\bar u} \in L^2(\omega \times (0, 1), H^1(D_1,\R^3))$ and $Z \in L^2 (\omega \times (0, 1),\R^3)$ such that

\begin{gather}
\label{ulim1}
u_{\e} \rightharpoonup u^- \quad \text{ weakly in } H^1(\Omega^-), \text{ strongly in } L^2(\Omega^-),\\
\label{ulim2}
u_{\e}(\cdot + \delta e_3) \rightharpoonup u^+ \quad \text{ weakly in } H^1(\Omega^+), \text{ strongly in } L^2(\Omega^+),
\end{gather}

\begin{gather}
\label{rc1}
\delta \mathcal T_{\e} (\widetilde{\mathcal R}_{\e}) \rightharpoonup \widetilde{\mathcal R} \quad \text{ weakly in } L^2(\omega, H^1 (0, 1)), \text{ such that }\\
\label{rc2}
\widetilde{\mathcal R} (x', 0) = \widetilde{\mathcal R} (x', 1) = 0, \text{ for a.e. } x' \in \omega,
\end{gather}

\begin{gather}
\label{uc2}
\frac{\delta}{r}(\mathcal T_{\e}(\widetilde{\mathcal U}_{\e,3}) - \widetilde{\mathcal U}_{\e,3}(\cdot, \cdot, 0)) \rightharpoonup \widetilde{\mathcal U}^{'}_3 \quad \text{ weakly in } L^2(\omega, H^1 (0, 1)),\\
\label{uc3}
\mathcal T_{\e}(\widetilde{\mathcal U}_{\e,3}) \rightharpoonup \widetilde{\mathcal U}_3 \quad \text{ weakly in } L^2(\omega, H^1 (0, 1)), \text{ such that }\\
\label{uc4}
\widetilde{\mathcal U}_3(\cdot, \cdot, \cdot) = \widetilde{\mathcal U}_3(\cdot, \cdot, 0) = {u_3^-}_{|\Sigma} = \widetilde{\mathcal U}_3( \cdot, \cdot, 1) = {u_3^+}_{|\Sigma}, \text{ a.e. in} \;\;  \omega\times (0,1),\\
\label{uc1}
\mathcal T_{\e}(\widetilde{\mathcal U}_{\e,\alpha}) \rightharpoonup \widetilde{\mathcal U}_{\alpha} \quad \text{ weakly in } L^2(\omega, H^1 (0, 1)), \quad \text{ for } \alpha = \overline{1, 2}, \text{ such that }\\
\label{uc5}
\widetilde{\mathcal U}_{\alpha}( \cdot, \cdot, 0) = {u_{\alpha}^-}_{|\Sigma}, \quad \widetilde{\mathcal U}_{\alpha}(\cdot, \cdot, 1) = {u_{\alpha}^+}_{|\Sigma} \text{ a.e. in } \; \omega,\\
\label{uc51}
\frac{\partial \widetilde{\mathcal U}_{\alpha}}{\partial X_3} ( \cdot, \cdot, 0) = \frac{\partial \widetilde{\mathcal U}_{\alpha}}{\partial X_3} (\cdot, \cdot, 1) = 0\text{ a.e. in } \; \omega,\\
\label{uc6}
\frac{\partial \widetilde{\mathcal U}_1}{\partial X_3} = \widetilde{\mathcal R}_2, \quad \frac{\partial \widetilde{\mathcal U}_2}{\partial X_3} = -\widetilde{\mathcal R}_1\text{ a.e. in } \; \omega\times(0,1),
\end{gather}

\begin{gather}
\label{limwarp}
\frac{\delta^2}{r^2}\mathcal T_{\e}^{'}(\widetilde{\bar u}_{\e}) \rightharpoonup \widetilde{\bar u} \quad \text{ weakly in } L^2(\omega \times (0, 1), H^1(D_1)),\\
\label{limwarp2}
\frac{\delta}{r}\mathcal T_{\e}^{'}(\widetilde{\bar u}_{\e}) \rightharpoonup 0 \quad \text{ weakly in } L^2(\omega, H^1(B_1)),
\end{gather} 

\begin{equation}\label{Z}
\frac{\delta}{r} \left( \frac{\partial \mathcal T_{\e} (\widetilde{\mathcal U}_{\e})}{\partial X_3} - \delta \mathcal T_{\e}(\widetilde{\mathcal R}_{\e} \wedge e_3) \right) \rightharpoonup Z \quad \text{ weakly in } L^2 (\omega \times (0, 1)).
\end{equation}
\end{proposition}

\begin{proof}
Convergences \eqref{ulim1}, \eqref{ulim2}, \eqref{rc1}, \eqref{uc2}, \eqref{uc3}, \eqref{uc1},\eqref{limwarp} and \eqref{Z} follow from the estimate \eqref{Eq327} and those in Lemma \ref{unfest}. 

\noindent Equalities \eqref{rc2} are the consequences of \eqref{Diff}$_1$-\eqref{Diff}$_2$.
To obtain \eqref{uc6} take into account that from \eqref{Z} we have
\begin{equation*}
\frac{\partial \widetilde{\mathcal U}}{\partial X_3} - \widetilde{\mathcal R} \wedge e_3 = \left(
\begin{array}{c}
\displaystyle \frac{\partial \widetilde{\mathcal U}_1}{\partial X_3} - \widetilde{\mathcal R}_2\\
\displaystyle \frac{\partial \widetilde{\mathcal U}_2}{\partial X_3} + \widetilde{\mathcal R}_1\\
\displaystyle \frac{\partial \widetilde{\mathcal U}_3}{\partial X_3}
\end{array}
\right)=0.
\end{equation*}
Then \eqref{rc2} yields \eqref{uc51}. Equations \eqref{uc4} are the consequences of $\displaystyle \frac{\partial \widetilde{\mathcal U}_3}{\partial X_3}=0$ and the estimates \eqref{estu1}, \eqref{estu2}.
Again due to \eqref{estu1}, \eqref{estu2}, we obtain 
\begin{gather*}
\widetilde{\mathcal U}_{\alpha} (x', 0) = {u_{\alpha}^-}_{|\Sigma}(x'),\qquad 
\widetilde{\mathcal U}_{\alpha} (x', 1) = {u_{\alpha}^+}_{|\Sigma}(x'),\quad \text{ for a.e. } x' \in \omega.
\end{gather*}
From Lemma \ref{unfest} we have $\ds \| \mathcal T_{\e}^{'}(\widetilde{\bar u}_{\e}) \|_{L^2(\omega, H^1( B_1))} \leq C \frac{r}{\delta}$ from  which and \eqref{limwarp} we deduce \eqref{limwarp2}.
\end{proof}

The strain tensor of the displacement $u_\e$ is
\begin{gather*}
\mathcal T_{\e}^{'}\left((\widetilde{\nabla  u_{\e}})_S\right)_{ij} = \mathcal T_{\e}^{'}((\widetilde{\nabla \bar u_{\e}})_S)_{ij}, \quad i, j = 1, 2,\\
\mathcal T_{\e}^{'}\left((\widetilde{\nabla u_{\e}})_S\right)_{13} = \ds \frac{1}{2}\left(\left(\frac{1}{\delta}\frac{\partial \mathcal T_{\e}(\widetilde{\mathcal U}_{\e,1})}{\partial X_3} - \mathcal T_{\e}(\widetilde{\mathcal R}_{\e,2})\right) - \frac{r}{\delta} \frac{\partial \mathcal T_{\e}(\widetilde{\mathcal R}_{\e,3})}{\partial X_3} X_2\right) + \mathcal T_{\e}^{'}((\widetilde{\nabla \bar u_{\e}})_S)_{13},\\
\mathcal T_{\e}^{'}\left((\widetilde{\nabla u_{\e}})_S\right)_{23} = \ds \frac{1}{2}\left(\left(\frac{1}{\delta}\frac{\partial \mathcal T_{\e} (\widetilde{\mathcal U}_{\e,2})}{\partial X_3} + \mathcal T_{\e}(\widetilde{\mathcal R}_{\e,1})\right) + \frac{r}{\delta}\frac{\partial \mathcal T_{\e}(\widetilde{\mathcal R}_{\e,3})}{\partial X_3} X_1\right) + \mathcal T_{\e}^{'}((\widetilde{\nabla \bar u_{\e}})_S)_{23},\\
\mathcal T_{\e}^{'}\left((\widetilde{\nabla u_{\e}})_S\right)_{33} = \ds \frac{1}{\delta}\frac{\partial \mathcal T_{\e}(\widetilde{\mathcal U}_{\e,3})}{\partial X_3} + \frac{r}{\delta} \frac{\partial \mathcal T_{\e}(\widetilde{\mathcal R}_{\e,1})}{\partial X_3} X_2  - \frac{r}{\delta}\frac{\partial \mathcal T_{\e}(\widetilde{\mathcal R}_{\e,2})}{\partial X_3} X_1 + \mathcal T_{\e}^{'}((\widetilde{\nabla \bar u_{\e}})_S)_{33}.
\end{gather*}
Define the field $\widetilde{\bar u'}\in L^2(\omega\times (0,1), H^1(D_1,\R^3))$ by
\begin{gather*}
\widetilde{\bar u'}_{\alpha} = \widetilde{\bar u}_{\alpha},\qquad \widetilde{\bar u'}_3 = \widetilde{\bar u}_3 + X_1 Z_1 + X_2 Z_2.
\end{gather*}
Then
\begin{gather*}
\frac{\partial \widetilde{\bar u'}_3}{\partial X_1} = \frac{\partial \widetilde{\bar u}_3}{\partial X_1} + Z_1,\qquad 
\frac{\partial \widetilde{\bar u'}_3}{\partial X_2} = \frac{\partial \widetilde{\bar u}_3}{\partial X_2} + Z_2.
\end{gather*}
As an immediate consequence of Proposition \ref{prop41}, we have
\begin{lemma}
\label{symmX}
There exist a symmetric matrix field $X \in L^2 (\omega \times B_1, \R^9)$ and a field \\
$\widetilde{\bar u'} \in L^2 (\omega \times (0,1), H^1 (D_1 , \R^3))$, such that

\begin{equation*}
\frac{\delta^2}{r} \mathcal T_{\e}^{'}\left((\widetilde{\nabla u_{\e}})_S\right) \rightharpoonup X \quad \text{ weakly in } L^2(\omega \times B_1, \R^9),
\end{equation*}
where $X$ is defined by
\begin{gather}
\label{beams-grad}
\begin{array}{l}
\ds X_{ij} =\frac{1}{2} \left(\frac{\partial \widetilde{\bar u}'_i}{\partial X_j} + \frac{\partial \widetilde{\bar u}'_j}{\partial X_i} \right), \quad i, j = 1, 2,\\
\ds X_{13} = X_{31} = \frac{1}{2}\left(\frac{\partial \widetilde{\bar u'}_3}{\partial X_1} - \frac{\partial \widetilde{\mathcal R}_3}{\partial X_3} X_2\right),\\
\ds X_{23} = X_{32} = \frac{1}{2}\left(\frac{\partial \widetilde{\bar u'}_3}{\partial X_2} + \frac{\partial \widetilde{\mathcal R}_3}{\partial X_3} X_1\right),\\
\ds X_{33} = \frac{\partial \widetilde{\mathcal U}_3'}{\partial X_3} - \frac{\partial^2 \widetilde{\mathcal U}_2}{\partial X_3^2} X_2  - \frac{\partial^2 \widetilde{\mathcal U}_1}{\partial X_3^2} X_1.
\end{array}
\end{gather}
\end{lemma}

\section{The limit problem}
\subsection{The equations for the domain $\Omega_{\e}^i$}
Denote by $\Theta$ the weak limit of the unfolded stress tensor $\ds\frac{\delta^2}{r} \mathcal T_{\e}^{'} (\sigma_{\e})$ in $L^2 (\omega \times B_1, \R^9)$:
$$\frac{\delta^2}{r} \mathcal T_{\e}^{'} (\sigma_{\e}) \rightharpoonup \Theta, \quad \text{ weakly in } L^2 (\omega \times B_1, \R^9).$$
Proceeding exactly as in Section 6.1 of \cite{griso2} and Section 8.1 of \cite{griso3}, we first derive $\widetilde{\bar u}'$ and this gives
\begin{gather*}
\widetilde{\bar u'}_1 = \nu^m \left( - X_1 \frac{\partial \widetilde{\mathcal U}^{'}_3}{\partial X_3} + \frac{X_1^2 - X_2^2}{2} \frac{\partial^2 \widetilde{\mathcal U}_1}{\partial X_3^2} + X_1 X_2 \frac{\partial^2 \widetilde{\mathcal U}_2}{\partial X_3^2} \right),\\
\widetilde{\bar u'}_2 = \nu^m \left( - X_2 \frac{\partial \widetilde{\mathcal U}^{'}_3}{\partial X_3} + X_1 X_2 \frac{\partial^2 \widetilde{\mathcal U}_1}{\partial X_3^2} + \frac{X_2^2 - X_1^2}{2} \frac{\partial^2 \widetilde{\mathcal U}_2}{\partial X_3^2} \right).
\end{gather*}
Similarly, the same computations as in Section 6.1 of \cite{griso2} lead to $\widetilde{\bar u'}_3 = 0$.

As a consequence of Lemma \ref{symmX} we obtain 
\begin{equation}
\label{theta}
\begin{aligned}
&\Theta_{11} = \Theta_{22} = \Theta_{12} = 0,&\\
&\Theta_{13} = -\mu^m X_2 \frac{\partial \widetilde{\mathcal R}_3}{\partial X_3}, \quad \Theta_{23} = \mu^m X_1 \frac{\partial \widetilde{\mathcal R}_3}{\partial X_3},&\\
&\Theta_{33} = E^m \Big( \frac{\partial \widetilde{\mathcal U}^{'}_3}{\partial X_3} - X_1\frac{\partial^2 \widetilde{\mathcal U}_1}{\partial X_3^2} - X_2 \frac{\partial^2 \widetilde{\mathcal U}_2}{\partial X_3^2} \Big).&
\end{aligned}
\end{equation}

\begin{proposition}\label{prop61}
 $(\widetilde{\mathcal U}_1, \widetilde{\mathcal U}_2)$ satisfy the variational formulation
\begin{equation}\label{pb62}
\begin{aligned}
 \frac{\pi \kappa_0^4}{4 \kappa_1^3} E^m\int_0^1 \frac{\partial^2 \widetilde{\mathcal U}_{\alpha}}{\partial X_3^2}(x', X_3) \frac{d^2 \varphi_{\alpha}}{d X_3^2}(X_3) \, dX_3 = \int_0^1 \widetilde{F}^m_{\alpha} (x',X_3)\varphi_{\alpha}(X_3)\, dX_3,\\
  \quad \forall \varphi_{\alpha} \in H_0^2 (0, 1),\qquad \hbox{for a.e. } x'\in \omega
\end{aligned}
\end{equation}
where
\begin{equation*}
\widetilde{F}^m_{\alpha} (x', X_3) = \int_{D_1} F^m_{\alpha}(x', X) dX_1 \, dX_2 \quad \text{a.e.  in } \omega \times (0, 1) \quad \alpha = 1,2.
\end{equation*} Furthermore $\widetilde{\mathcal R}_3=0$ and there exists $a\in L^2(\omega)$ such that
$$\widetilde{\mathcal U}^{'}_3(x',X_3)=a(x')X_3\qquad \text{a.e.  in } \omega \times (0, 1) .$$
\end{proposition}
\begin{proof}
\textit{Step 1.} We obtain the limit equations in $\Omega_{\e}^i$.

We will use the following test function:
\begin{equation*}
v_{\e} (x) = \frac{r}{\delta}\psi (\varepsilon \xi)\left(
\begin{array}{c}
\ds \frac{\delta}{r}\varphi_1 \left( \frac{x_3}{\delta}\right) - \frac{x_2 - \e \xi_2}{r} \varphi_4 \left( \frac{x_3}{\delta} \right)\\[2mm]
\ds \frac{\delta}{r}\varphi_2 \left( \frac{x_3}{\delta}\right) + \frac{x_1 - \e \xi_1}{r} \varphi_4 \left( \frac{x_3}{\delta} \right)\\[2mm]
\ds \varphi_3 \left(\frac{x_3}{\delta}\right) - \frac{x_1 - \e \xi_1}{r} \frac{d \varphi_1}{d X_3} \left( \frac{x_3}{\delta}\right) - \frac{x_2 - \e \xi_2}{r} \frac{d \varphi_2}{d X_3} \left(\frac{x_3}{\delta}\right) 
\end{array}
\right),\qquad \ds\xi = \left[ \frac{x'}{\e} \right]_Y,
\end{equation*}
where $\psi \in C^{\infty}_c (\omega)$, $\varphi_3$ and $\varphi_4  \in H_0^1 (0, 1)$, $\varphi_{1}$ and  $\varphi_{2} \in H_0^2 (0, 1)$.
Computation of the symmetric strain tensor gives
\begin{equation*}
(\nabla v_{\e})_S = \frac{r}{\delta^{2}}\psi(\e \xi)\left(
\begin{array}{ccc}
0 & 0 & \ds -\frac{1}{2} \frac{x_2 - \e \xi_2}{r } \frac{d \varphi_4}{d X_3}\\[2mm]
\ldots & 0 & \ds \frac{1}{2}\frac{x_1 - \e \xi_1}{r  } \frac{d \varphi_4}{d X_3} \\[2mm]
\ldots & \ldots & \ds   \left( \frac{d \varphi_3}{d X_3} - \frac{x_1 - \e \xi_1}{r} \frac{d^2 \varphi_1}{d X_3^2} - \frac{x_2 - \e \xi_2}{r} \frac{d^2 \varphi_2}{d X_3^2} \right)\\
\end{array}
\right)\qquad \hbox{in } \e\xi+B_1.
\end{equation*}
Then
\begin{equation*}
\frac{\delta^2}{r} \mathcal T_{\e}^{'}((\widetilde{\nabla v_{\e}})_S) \rightarrow \psi(x')\left(
\begin{array}{ccc}
0 & 0 & \ds -\frac{1}{2} X_2 \frac{d \varphi_4}{d X_3}\\[2mm]
\ldots & 0 & \ds \frac{1}{2}X_1 \frac{d \varphi_4}{d X_3} \\[2mm]
\ldots & \ldots & \ds\frac{d \varphi_3}{d X_3} - X_1 \frac{d^2 \varphi_1}{d X_3^2} - X_2 \frac{d^2 \varphi_2}{d X_3^2}\\
\end{array}
\right) = V(x', X) \quad \text{ strongly in } L^2 (\omega \times B_1).
\end{equation*}
Moreover,
\begin{equation*}
\mathcal T_{\e}^{'} (\widetilde{v}_{\e}) \rightarrow \psi (x')\left(
\begin{array}{c}
\varphi_1 (X_3)\\
\varphi_2 (X_3)\\
0
\end{array}
\right) \quad \text{ strongly in } L^2 (\omega \times B_1).
\end{equation*}
Unfolding the integral over $\Omega_{\e}^i$ yields 
\begin{equation*}
\begin{aligned}
\int_{\Omega^i_{\e}} \sigma_{\e} : (\nabla v_{\e})_S dx &= \sum_{\xi \in \Xi_{\e}} \int_{\e \xi + B_{r, \delta}} \sigma_{\e} : (\widetilde{\nabla v_{\e}})_S dx  \\
&= r^2 \delta \sum_{\xi \in \Xi_{\e}} \int_{B_1} \mathcal T_{\e}^{'} (\sigma_{\e}) : \mathcal T_{\e}^{'}((\widetilde{\nabla v_{\e}})_S) dx' \, dX_1 \, dX_2 \, dX_3  \\
&= \frac{r^2 \delta}{\e^2}\int_{\omega \times B_1} \mathcal T_{\e}^{'} (\sigma_{\e}) : \mathcal T_{\e}^{'}((\widetilde{\nabla v_{\e}})_S) dx' \, dX_1 \, dX_2 \, dX_3.
\end{aligned}
\end{equation*}
In the same way for the integral involving the forces we get
\begin{equation*}
\int_{\Omega^i_{\e}} f_{\e}\cdot v_\e dx = \frac{r^2 \delta}{\e^2}\int_{\omega \times B_1} \mathcal T_{\e}^{'} (f_{\e})\cdot \mathcal T_{\e}^{'}(\widetilde{v}_{\e}) dx' \, dX_1 \, dX_2 \, dX_3. 
\end{equation*}
Passing to the limit gives
\begin{equation}
\label{unflim}
\frac{\kappa_0^4}{\kappa_1^3}\int_{\omega \times B_1} \Theta : V dx' \, dX = \sum_{\alpha=1}^2 \int_{\omega \times B_1}  F^m_\alpha (x', X) \psi(x')\varphi_\alpha( X) dx' \, dX.
\end{equation}
We can localize the above equation. Hence
\begin{multline}
\label{beamloc}
\frac{\pi \kappa_0^4}{4 \kappa_1^3}\mu^m \int_{\omega \times (0, 1)} \frac{\partial \widetilde{\mathcal R}_3}{\partial X_3} \frac{d \varphi_4}{d X_3} \psi  \, dx' \, dX_3 + \frac{\pi \kappa_0^4}{4 \kappa_1^3} E^m\int_{\omega \times (0, 1)} \left( 4 \frac{\partial \widetilde{\mathcal U}^{'}_3}{\partial X_3} \frac{d \varphi_3}{d X_3} + \frac{\partial^2 \widetilde{\mathcal U}_1}{\partial X_3^2} \frac{d^2 \varphi_1}{d X_3^2} + \frac{\partial^2 \widetilde{\mathcal U}_2}{\partial X_3^2} \frac{d^2 \varphi_2}{d X_3^2} \right) \psi \, dx' \, dX_3 \\
 = \int_{\omega \times (0, 1)} \left(\widetilde{F}^m_1 \varphi_1 +  \widetilde{F}^m_2 \varphi_2 \right) \psi dx' \, dX_3.
\end{multline}
The density of the tensor
product ${\cal C}^\infty_c(\omega )\otimes H^1_0(0,1)$ (resp. ${\cal C}^\infty_c(\omega )\otimes H^2_0(0,1)$) in  $L^2(\omega  ; H^1_0 (0,1))$ (resp. $L^2(\omega  ; H^2_0 (0,1))$)  implies 
\begin{equation}\begin{aligned}
\label{Beamloc}
&\frac{\pi \kappa_0^4}{4 \kappa_1^3}\mu^m \int_{\omega \times (0, 1)} \frac{\partial \widetilde{\mathcal R}_3}{\partial X_3} \frac{\partial \Phi_4}{\partial  X_3}  \, dx' \, dX_3 + \frac{\pi \kappa_0^4}{4 \kappa_1^3} E^m\int_{\omega \times (0, 1)} \left( 4 \frac{\partial \widetilde{\mathcal U}^{'}_3}{\partial X_3} \frac{\partial \Phi_3}{\partial X_3} + \frac{\partial^2 \widetilde{\mathcal U}_1}{\partial X_3^2} \frac{\partial^2 \Phi_1}{\partial X_3^2} + \frac{\partial^2 \widetilde{\mathcal U}_2}{\partial X_3^2} \frac{\partial^2 \Phi_2}{\partial X_3^2} \right) \, dx' \, dX_3 \\
& = \int_{\omega \times (0, 1)} \left(\widetilde{F}^m_1 \Phi_1 +  \widetilde{F}^m_2 \Phi_2 \right)  dx' \, dX_3\qquad \forall \Phi_3, \; \Phi_4\in L^2(\omega ; H^1_0(0,1)),\quad \forall \Phi_1, \; \Phi_2\in L^2(\omega ; H^2_0(0,1)).
\end{aligned}\end{equation}

\noindent \textit{Step 2.} We obtain $\widetilde{\mathcal R}_3$, $\widetilde{\mathcal U}^{'}_3$.

\noindent Since $\varphi_3\in H_0^1 (0, 1) $ is not in the right-hand side of the equation \eqref{beamloc} we obtain
\begin{equation}
\label{u3}
E^m \int_0^1 \frac{\partial \widetilde{\mathcal U}'_3}{\partial X_3} \frac{d \varphi_3}{d X_3} dX_3 = 0\quad \Rightarrow \quad \frac{\partial^2 \widetilde{\mathcal U}'_3}{\partial X_3^2} = 0 \quad \text{  a.e. in } \;\;  \omega\times (0,1).
\end{equation}
Moreover, we have $\widetilde{\mathcal U}'_3(x', 0)=0$ for a.e. $x' \in \omega$. Therefore, there exists $a\in L^2(\omega)$ such that
\begin{equation*}
 \widetilde{\mathcal U}'_3 (x', X_3)= X_3 a(x'), \quad \hbox{for a.e. } \; (x', X_3)\in \omega\times (0,1).
 \end{equation*}
Similarly, recalling that $\varphi_4\in H_0^1 (0, 1) $ and taking $\varphi_1 = \varphi_2 = \varphi_3 = 0$ in \eqref{beamloc} lead to
\begin{equation*}
 \mu^m \int_0^1 \frac{\partial \widetilde{\mathcal R}_3}{\partial X_3} \frac{d \varphi_4}{d X_3} \, dX_3 = 0\quad \Rightarrow \quad \frac{\partial^2 \widetilde{\mathcal R}_3}{\partial X_3^2} = 0 \quad \text{  a.e. in } \;\; \omega\times (0,1),
\end{equation*}
which together with the boundary conditions \eqref{rc2} (see Proposition \ref{limits_case1}) gives
$\widetilde{\mathcal R}_3 = 0$.
\end{proof}
The variational problem \eqref{pb62} and the boundary conditions \eqref{uc5}-\eqref{uc51} allow to determine ${\mathcal U}_\alpha$ ($\alpha=1,2$) in terms of the applied forces $\widetilde{F}^m_{\alpha} $ and the traces $u^{\pm}_{\alpha|\Sigma}$.

\subsection{The equations for the macroscopic domain}
Denote
$$
\begin{aligned}
&\V=\Big\{ v\in L^2(\Omega^-\cup\O^+ ; \R^3)\;|\; v_{|\Omega^-}\in H^1(\O^- ; \R^3)\enskip \hbox{and}\enskip v_{|\Omega^-}=0\;\;\hbox{on}\;\Gamma,\\
&\hskip 4.5cm v_{|\Omega^+}\in H^1(\O^+ ; \R^3)\enskip \hbox{and}\enskip v_{3|\Omega^+} =v_{3|\Omega^-} \;\;\hbox{on}\;\Sigma\Big\}\\
&\V_T=\Big\{ (v,{\cal V}_1, {\cal V}_2, {\cal V}_3, {\cal V}_4)\in \V\times [L^2(\Omega; H^2(0,1))]^2\times [L^2(\Omega; H^1(0,1))]^2\;|\;  \\
&{\mathcal V}_{\alpha}( \cdot, \cdot, 0) = {v_{\alpha}^-}_{|\Sigma}, \quad {\mathcal V}_{\alpha}(\cdot, \cdot, 1) = {v_{\alpha}^+}_{|\Sigma} \text{ a.e. in } \; \omega,\\
&{\mathcal V}_3( \cdot, \cdot, 0) ={\mathcal V}_4( \cdot, \cdot, 0) ={\mathcal V}_4( \cdot, \cdot, 1) =\frac{\partial {\mathcal V}_{\alpha}}{\partial X_3} ( \cdot, \cdot, 0) = \frac{\partial {\mathcal V}_{\alpha}}{\partial X_3} (\cdot, \cdot, 1) = 0\text{ a.e. in } \; \omega,\quad \alpha\in\{1,2\} \Big\}
\end{aligned}
$$

\noindent Let $\chi$ be in ${\cal C}^\infty_c(\R^2)$ such that $\chi(y)=1$ in   $D_1$ (the disc centered in $O=(0,0)$ and radius 1).
\medskip
\smallskip

{\it From now on we only consider the case (ii).}

\subsubsection{Determination of $\widetilde{\mathcal U}^{'}_3$}
\label{xi-det}
\begin{lemma}\label{lem62} The function $a$ introduced in Proposition \ref{prop61} is equal to 0 and
$$\widetilde{\mathcal U}^{'}_3(x',X_3)=0 \qquad \text{a.e.  in } \omega \times (0, 1) .$$
\end{lemma}
\begin{proof} For any $\psi_3 \in {\cal C}^1 (\overline{\omega} \times [0, 1])$ satisfying $\psi_3(x',0)=0$ for  every $x'\in \omega$, we consider the following test function: 
$$
\begin{aligned}
v_{\e,\alpha}(x)& =0\qquad \hbox{for a.e. } x\in \O_\e ,\;\; \alpha=1,2,\\
v_{\e}(x) &= 0 \qquad \hbox{for a.e. } x\in \O^- ,\\
v_{\e,3}(x) &= \frac{r}{\delta}\Big[\psi_3(x',1)\Big(1 -\chi \Big( \frac{\e}{r}\Big\{\frac{x' }{\e}\Big\}_Y\Big)\Big) + \psi_3\Big(\e \Big[\frac{x' }{\e} \Big]_Y, 1\Big)\chi \Big( \frac{\e}{r}\Big\{\frac{x' }{\e}\Big\}_Y\Big)\Big], \qquad \hbox{for a.e. } x\in  \O^+_\e,\\
v_{\e,3}(x) & = \frac{r}{\delta} \psi_3 \Big( \e \Big[\frac{x' }{\e} \Big]_Y, \frac{x_3}{\delta} \Big), \qquad \hbox{for a.e. } x\in  \O^i_\e.
\end{aligned}$$
If $\ds \frac{r}{\e}$ is small enough, $v_\e$ is an admissible test function. The symmetric strain tensor in $\Omega^i_{\e}$ is given by
\begin{equation*}
(\nabla v_{\e})_S = \frac{r}{\delta^{2}}\left(
\begin{array}{ccc}
0 & 0 & 0 \\[2mm]
\ldots & 0 & 0 \\[2mm]
\ldots & \ldots & \ds \frac{\partial \psi_3}{\partial X_3} \Big(\e\xi, \frac{x_3}{\delta}\Big)\\
\end{array}
\right) \qquad \hbox{a.e. in } \e\xi+B_{r,\d}.
\end{equation*}
Then
\begin{equation*}
\frac{\delta^2}{r} \mathcal T_{\e}^{'}((\widetilde{\nabla v_{\e}})_S) \rightarrow \left(
\begin{array}{ccc}
0 & 0 &0 \\[2mm]
\ldots & 0 & 0 \\[2mm]
\ldots & \ldots & \ds \frac{\partial \psi_3}{\partial X_3} (x',X_3)\\
\end{array}
\right) = V(x', X) \quad \text{ strongly in } L^2 (\omega \times B_1).
\end{equation*}
Elements of the symmetric strain tensor in $\Omega_{\e}^+$ are written as follows:
\begin{equation*}
\begin{aligned}
&(\nabla v_{\e})_S^{11} =  (\nabla v_{\e})_S^{22} =  (\nabla v_{\e})_S^{12} = (\nabla v_{\e})_S^{33} =0,\\  
&(\nabla v_{\e})_S^{\alpha 3} = (\nabla v_{\e})_S^{3\alpha} = \frac{1}{2}\frac{r}{\delta} \frac{\partial \psi_3}{\partial x_\alpha} (x',1)  (1 - \chi (y))  + \frac{1}{2\delta} \frac{\partial \chi}{\partial y_\alpha}(y) \left(\psi_3(x',1) - \psi_3\left(\e\left[\frac{x'}{\e} \right]_Y, 1\right)\right),
\end{aligned}
\end{equation*}
where $\ds y = \frac{\e}{r}\Big\{\frac{x'}{\e}\Big\}_Y$.

\medskip

By Lemma \ref{lem61} (see Appendix) and taking into account $\ds \frac{r}{\delta} \rightarrow 0$, the following convergences hold:
\begin{equation*}
\begin{aligned}
&v_{\e}(\cdot + \delta e_3) \longrightarrow 0 \quad \text{ strongly in } H^1(\Omega^+;\R^3),\\
&(\nabla v_{\e})_S \longrightarrow 0 \quad \text{ strongly in } L^2 (\Omega^+ ; \R^9).
\end{aligned}
\end{equation*}
Moreover,
\begin{gather*}
\mathcal T_{\e}' (v_{\e}) \longrightarrow 0 \quad \text{ strongly in } H^1 (\omega \times B_1 ; \R^3).
\end{gather*}
Using $v_{\e}$ as a test function in \eqref{wf} and passing to the limit in the unfolded formulation give
\begin{equation*}
\int_{\omega \times (0, 1)} \frac{\partial \widetilde{\mathcal U}'_3}{\partial X_3}(x',X_3) \frac{\partial \psi_3}{\partial X_3}(x', X_3) \, dx' \, dX =\int_{\omega \times (0, 1)} a(x') \frac{\partial \psi_3}{\partial X_3}(x', X_3) \, dx' \, dX = 0.
\end{equation*}
Hence $a = 0$. Since the test functions are dense in 
$$\V_s=\Big\{\Psi\in L^2(\omega; H^1(0,1))\;|\; \Psi(x',0)=0\;\; \text{a.e. in } \omega\Big\}$$ we obtain
\begin{equation}\label{EQ. 6.7}
\int_{\omega \times (0, 1)} \frac{\partial \widetilde{\mathcal U}'_3}{\partial X_3}(x',X_3) \frac{\partial \Psi}{\partial X_3}(x', X_3) \, dx' \, dX = 0\qquad \forall \Psi\in \V_s.
\end{equation}

\end{proof}
As a consequence of the above Lemma and Proposition \ref{prop61} one gets 
\begin{equation}
\label{theta-bis}
\begin{aligned}
&\Theta_{ij} = 0,\qquad (i,j)\not=(3,3)\\
&\Theta_{33} = -E^m \Big(X_1\frac{\partial^2 \widetilde{\mathcal U}_1}{\partial X_3^2} + X_2 \frac{\partial^2 \widetilde{\mathcal U}_2}{\partial X_3^2} \Big).
\end{aligned}
\end{equation}
\subsubsection{Determination of $u^\pm_\alpha$ and $u_3$}\label{SSS622}
\begin{theorem}\label{TH61}
The variational formulation of the limit problem for \eqref{wf} is
\begin{equation}
\begin{aligned}\label{limf}
&\int_{\Omega^+ \cup \Omega^-} \sigma^{\pm} \, : (\nabla v)_S \, dx + \frac{\pi \kappa_0^4}{4 \kappa_1^3} E^m \int_{\omega \times (0, 1)} \sum_{\alpha = 1}^2 \frac{\partial^2 \widetilde{\mathcal U}_{\alpha}}{\partial X_3^2}\frac{\partial^2 \psi_{\alpha}}{\partial X_3^2} \, dx' \, dX_3\\
+&\frac{\pi \kappa_0^4}{4 \kappa_1^3}\mu^m \int_{\omega \times (0, 1)} \frac{\partial \widetilde{\mathcal R}_3}{\partial X_3} \frac{\partial \psi_4}{\partial  X_3}  \, dx' \, dX_3 + \frac{\pi \kappa_0^4}{\kappa_1^3} E^m\int_{\omega \times (0, 1)}  \frac{\partial \widetilde{\mathcal U}^{'}_3}{\partial X_3} \frac{\partial \Phi_3}{\partial X_3}  \, dx' \, dX_3 \\
= &\int_{\Omega^+ \cup \Omega^-} F \,v \,dx + \int_{\omega \times (0, 1)} \sum_{\alpha = 1}^2 \widetilde{F}_{\alpha}^m \psi_{\alpha} dx' \, dX_3 + \int_{\omega} \overline{F}_3^m v_3 dx',\\
& \qquad \forall (v, \psi_1, \psi_2, \psi_3, \psi_4)\in \V_T, \qquad \forall \Phi_3 \in L^2(\omega ; H^1_0(0,1)),
\end{aligned}\end{equation}
where
$$
\overline{F}_3^m (x') = \int_{B_1} F_3^m (x', X) \, dX, \quad x' \in \omega.
$$
\end{theorem}
\begin{proof}
For any $v\in \V$ such that $v_{|\Omega^-}\in W^{1,\infty}(\O^- , \R^3)$ and $v_{|\Omega^+}\in W^{1,\infty}(\O^+ , \R^3)$, we first define the displacement $v_{\e,r}$ in the following way: 
\begin{equation}
\label{corrg}
v_{\e,r} (x) = v(x)\Big(1 -\chi \Big( \frac{\e}{r}\Big\{\frac{x' }{\e}\Big\}_Y\Big)\Big) + v\Big(\e \Big[\frac{x'}{\e} \Big]_Y, x_3\Big)\chi \Big( \frac{\e}{r}\Big\{\frac{x' }{\e}\Big\}_Y\Big), \qquad \hbox{for a.e. } x \in \O^-\cup\O^+.
\end{equation}
Then denote $h$ the following function belonging to $W^{1,\infty} (-L,L)$
\begin{equation}
\label{h}
h (x_3) = \left\{
\begin{aligned}
 &{x_3+L\over L}, & \quad &x_3 \in [- L, 0],\\
 &1, & \quad &x_3\geq 0.
\end{aligned}
\right.
\end{equation}
Now consider the test displacement 
$$
\begin{aligned}
v'_{\e}(x) &= v(x)\big(1-h(x_3)\big)+ v_{\e,r} (x) h(x_3), \qquad \hbox{for a.e. } x\in \O^- ,\\
v'_{\e}(x) &= v_{\e,r} (x', x_3-\d), \qquad \hbox{for a.e. } x \in \O^+_\e,\\
v'_{\e}(x) & = \left(
\begin{array}{c}
 \ds \psi _1 \Big( \e \Big[\frac{x' }{\e} \Big]_Y, \frac{x_3}{\delta} \Big)\\[2mm]
 \ds \psi _2 \Big( \e \Big[\frac{x' }{\e} \Big]_Y, \frac{x_3}{\delta} \Big)\\[2mm]
 \ds v_3 \Big(\e \Big[\frac{x' }{\e} \Big]_Y, 0\Big)-{\e\over \d}\Big\{\frac{x' }{\e} \Big\}_Y\cdot \frac{\partial \psi}{\partial X_3} \Big( \e \Big[\frac{x' }{\e} \Big]_Y, \frac{x_3}{\delta} \Big)
\end{array} 
\right)  \quad \hbox{for a.e. } x\in  \O^i_\e,
\end{aligned}$$
where $\psi_{\alpha}  \in {\cal C}^1 (\overline{\omega} ; {\cal C}^3([0,1])), \alpha= 1,2,$ satisfies
\begin{equation*}
\psi_{\alpha}  (x', 0) = v_{\alpha|\Omega^-} (x', 0), \quad \psi_{\alpha}  (x', 1) = v_{\alpha|\Omega^+} (x', 0)\qquad \text{ for every } x'\in \omega.
\end{equation*}
If $\ds \frac{r}{\e}$ is small enough, $v'_\e$ is an admissible test displacement. 
\medskip

\noindent Then by Lemma \ref{lem61} the following convergences hold:
\begin{gather*}
\begin{aligned}
&v'_{\e}(\cdot + \delta e_3) \longrightarrow v \quad \text{ strongly in } H^1(\Omega^+;\R^3),\\
&v'_{\e} \longrightarrow v \quad \text{ strongly in } H^1(\Omega^-;\R^3),\\
&(\nabla v'_{\e})_S \longrightarrow (\nabla v)_S \quad \text{ strongly in } L^2 (\Omega^+ \cup \Omega^- ; \R^9).
\end{aligned}
\end{gather*}
Computation of the strain tensor in $\Omega_{\e}^i$ gives
\begin{equation*}
\begin{aligned}
  (\nabla v'_{\e})_S^{ij} &=0\qquad (i,j)\not=(3,3),\\
  (\nabla v'_{\e})_S^{33} &= - \frac{r}{\delta^2} \left( X_1 \frac{\partial^2 \psi_1}{\partial X_3^2}\Big( \e \Big[\frac{x' }{\e} \Big]_Y, X_3\Big) + X_2 \frac{\partial^2 \psi_2}{\partial X_3^2} \Big( \e \Big[\frac{x' }{\e} \Big]_Y, X_3\Big) \right).
\end{aligned}\end{equation*}
Therefore,
\begin{gather*}
\mathcal T_{\e}' (v'_{\e})\longrightarrow \left(
\begin{array}{c}
 \psi_1 (x', X_3)\\
 \psi_2 (x', X_3)\\
 v_3 (x', 0)
\end{array}
\right) \quad \text{ strongly in } L^2(\omega \times B_1; \R^3),\\
\frac{\delta^2}{r} \mathcal T_{\e}' \Big((\nabla v'_{\e})^{33}_S\Big) \longrightarrow - \left( X_1 \frac{\partial^2 \psi_1}{\partial X_3^2} (x', X_3) + X_2 \frac{\partial^2 \psi_2}{\partial X_3^2} (x', X_3)\right)  \quad \text{ strongly in } L^2(\omega \times B_1).
\end{gather*}
Unfolding and passing to the limit in \eqref{wf} give
\begin{multline*}
\int_{\Omega^{\pm}} \sigma^{\pm} : (\nabla v)_S \, dx - \frac{\kappa_0^4}{\kappa_1^3}\int_{\omega \times B_1} \Theta : \left( X_1 \frac{\partial^2 \psi_1}{\partial X_3^2} + X_2 \frac{\partial^2 \psi_2}{\partial X_3^2} \right) dx' \, dX = \\
= \int_{\Omega^{\pm}} F \, v \, dx + \int_{\omega \times B_1} (F^m_1 \psi_1 + F^m_2 \psi_2 + F_3 v_3)\, dx' \, dX.
\end{multline*}
Since the space $W^{1,\infty}(\O^+ ; \R^3)$ is dense in $H^1(\O^+;\R^3)$, the space of functions in $W^{1,\infty}(\O^- , \R^3)$ vanishing on $\Gamma$ is dense in $H^1(\O^-;\R^3)$ and the space ${\cal C}^1 (\overline{\omega} ; {\cal C}^3([0,1]))$ is dense in $L^2(\omega; H^1(0,1))$, the above equality holds for every $v$ in $\V$ and every $ \psi_1$, $\psi_2$ in $L^2(\omega  ; H^1(0,1))$ satisfying 
$$\psi_{\alpha}  (x', 0) = v_{\alpha|\Omega^-} (x', 0), \quad \psi_{\alpha}  (x', 1) = v_{\alpha|\Omega^+} (x', 0)\qquad \text{ for a.e. } x'\in \omega.$$
Finally,  integrating over $D_1$ and due to \eqref{Beamloc}, \eqref{EQ. 6.7} and \eqref{theta-bis} we obtain the result.
\end{proof}

\subsubsection{The case (i)}
We introduce the classical unfolding operator.
\begin{definition}
For $\varphi$ Lebesgue-measurable function on $\omega$, the unfolding operator $\mathcal T''_{\e}$ is defined as follows:
\begin{equation*}
\mathcal T''_{\e} (\varphi) (s, y) = 
\left\{
\begin{array}{ll}
\ds \varphi \Big(\e\Big[{s\over \e}\Big]_Y+\e y\Big), & \text{ for a.e.  }\;\;  (s, y) \in \widehat \omega_{\e} \times Y,\\
 0, & \text{ for a.e. } \;\; (s, y) \in \Lambda_{\varepsilon} \times Y.
\end{array}
\right.
\end{equation*}
\end{definition}

Recall that (see \cite{cdg3})
\begin{lemma}\label{lem63} Let $\phi$ be in $W^{1,\infty}(\omega)$ and $\phi_\e$ defined by
$$
\phi_\e(x')=\chi\Big(\Big\{{x'\over \e}\Big\}_Y\Big)\phi\Big(\e\Big[{x'\over \e}\Big]_Y\Big)+\Big[1-\chi\Big(\Big\{{x'\over \e}\Big\}_Y\Big)\Big]\phi(x')\quad \hbox{for a.e. }\;x'\in \omega.
$$
Then we have
$$
\begin{aligned}
&\mathcal T''_{\e} (\phi_\e)\longrightarrow \phi\quad \hbox{strongly in } \; L^2(\omega; H^1(Y)),\\
&\mathcal T''_{\e} (\nabla \phi_\e)\longrightarrow \nabla \phi\quad \hbox{strongly in } \; L^2(\omega\times Y).
\end{aligned}
$$ 
\end{lemma}
\begin{theorem}
The variational formulation for the problem \eqref{wf} in the case $(i)$
\begin{equation}
\label{limwf2}
\begin{aligned}
&\int_{\Omega^+ \cup \Omega^-} \sigma^{\pm} : (\nabla v)_S \, dx + \frac{\pi \kappa_0^4}{4 \kappa_1^3} E^m \int_{\omega \times (0, 1)} \sum_{\alpha = 1}^2 \frac{\partial^2 \psi_{\alpha}}{\partial X_3^2} \frac{\partial^2 \widetilde{\mathcal U}_{\alpha}}{\partial X_3^2} \, dx' dX_3 \\
+&\frac{\pi \kappa_0^4}{4 \kappa_1^3}\mu^m \int_{\omega \times (0, 1)} \frac{\partial \widetilde{\mathcal R}_3}{\partial X_3} \frac{\partial \psi_4}{\partial  X_3}  \, dx' \, dX_3 + \frac{\pi \kappa_0^4}{\kappa_1^3} E^m\int_{\omega \times (0, 1)}  \frac{\partial \widetilde{\mathcal U}^{'}_3}{\partial X_3} \frac{\partial \Phi_3}{\partial X_3}  \, dx' \, dX_3\\
&= \int_{\Omega^+ \cup \Omega^-} F \, v \, dx + \int_{\omega \times (0, 1)} \sum_{\alpha = 1}^2 \widetilde{F}_{\alpha}^m \psi_{\alpha} dx' \, dX_3 + \int_{\omega} \overline{F}_3^m v_3 dx',\\
& \qquad \forall (v, \psi_1, \psi_2, \psi_3, \psi_4)\in \V_T, \qquad \forall \Phi_3 \in L^2(\omega ; H^1_0(0,1)).
\end{aligned}
\end{equation}
\end{theorem}
\begin{proof}
\textit{Step 1.} Pass to the limit in the weak formulation.

To \eqref{ulim1} and \eqref{ulim2}  we add 
\begin{gather}
\label{ulim1-bis}
\mathcal T''_{\e} (u_{\e}) \rightharpoonup u^- \quad \text{ weakly in } L^2(\Omega^-; H^1(Y)),\\
\mathcal T''_{\e} (\nabla u_{\e}) \rightharpoonup \nabla u^- +\nabla_y \widehat{u}^- \quad \text{ weakly in } L^2(\Omega^-\times Y),\\
\label{ulim2-bis}
\mathcal T''_{\e} (u_{\e})(\cdot + \delta e_3,\cdot\cdot) \rightharpoonup u^+ \quad \text{ weakly in } L^2(\Omega^+; H^1(Y)),\\
\mathcal T''_{\e} (\nabla u_{\e})(\cdot + \delta e_3,\cdot\cdot)  \rightharpoonup \nabla u^+ +\nabla_y \widehat{u}^+ \quad \text{ weakly in } L^2(\Omega^+\times Y),
\end{gather}
where $\widehat{u}^- $ belongs to $L^2(\Omega^-; H^1_{per}(Y;\R^3))$ and $\widehat{u}^+ $ belongs to $L^2(\Omega^+; H^1_{per}(Y;\R^3))$.
\begin{remark}
Here the third variable of $u_{\e}$ is considered as a parameter, on which the unfolding operator $\mathcal T''_{\e}$ does not have any effect.
\end{remark}

\noindent{\it Step 2. } Determination of ${\cal U}'_3$.
\medskip

To determine the function $a$ introduced in Proposition \ref{prop61}, take $\psi_3 \in {\cal C}^1 (\overline{\omega} \times [0, 1])$ satisfying $\psi_3(x',0)=0$ for  every $x'\in \omega$ and consider the following test function: 
$$
\begin{aligned}
v_{\e,\alpha}(x)& =0\qquad \hbox{for a.e. } x\in \O_\e ,\;\; \alpha=1,2,\\
v_{\e}(x) &= 0 \qquad \hbox{for a.e. } x\in \O^- ,\\
v_{\e,3}(x) &= \e^{1/3} \Big[\psi_3(x',1)\Big(1 -\chi \Big(\Big\{\frac{x' }{\e}\Big\}_Y\Big)\Big) + \psi_3\Big(\e \Big[\frac{x'}{\e} \Big]_Y, 1\Big)\chi \Big(\Big\{\frac{x' }{\e}\Big\}_Y\Big)\Big] \qquad \hbox{for a.e. } x\in  \O^+_\e,\\
v_{\e, 3}(x) & = \e^{1/3} \psi_3 \Big(\e\Big[\frac{x'}{\e}\Big]_Y, \frac{x_3}{\e^{2/3}}\Big) \chi \Big(\Big\{\frac{x'}{\e}\Big\}_Y\Big) \qquad \hbox{for a.e. } x\in  \O^i_\e.
\end{aligned}$$
We obtain the following convergences:
\begin{gather*}
\begin{aligned}
&v_{\e}(\cdot + \delta e_3) \longrightarrow 0 \quad \text{ strongly in } H^1(\Omega^+ \cup \Omega^-;\R^3),\\
&(\nabla v_{\e})_S \longrightarrow 0 \quad \text{ strongly in } L^2 (\Omega^+ \cup \Omega^-; \R^9),
\end{aligned}\\
\mathcal T_{\e}' (v_{\e}) \longrightarrow 0 \quad \text{ strongly in } H^1 (\omega \times B_1 ; \R^3).
\end{gather*}
Unfolding and passing to the limit as in the Subsection \ref{xi-det} we obtain that $a = 0$.
\medskip

\noindent{\it Step 3. }
For any $v\in \V$ such that $v_{|\Omega^-}\in W^{1,\infty}(\O^-; \R^3)$ and $v_{|\Omega^+}\in W^{1,\infty}(\O^+; \R^3)$, define the displacement $v_\e$ in the following way: 
\begin{equation}
\label{corrg(i)}
v_{\e} (x) = v(x)\Big(1 -\chi \Big(\Big\{\frac{x' }{\e}\Big\}_Y\Big)\Big) + v\Big(\e \Big[\frac{x' }{\e} \Big]_Y, x_3\Big)\chi \Big(\Big\{\frac{x' }{\e}\Big\}_Y\Big), \qquad \hbox{for a.e. } x \in \O^-\cup\O^+.
\end{equation}

Consider the following test displacement:
$$
\begin{aligned}
v'_{\e}(x) &= v(x)\big(1-h(x_3)\big)+v_{\e}(x) h(x_3) + \e \Psi^{(-)}(x',x_3)\widehat v \left( \left\{\frac{x'}{\e}\right\}_Y\right) , \qquad \hbox{for a.e. } x\in \O^- ,\\
v'_{\e}(x) &= v_{\e}(x', x_3-\d) + \e \Psi^{(+)}(x',x_3-\d)\widehat v \left( \left\{\frac{x'}{\e}\right\}_Y\right), \qquad \hbox{for a.e. } x \in \O^+_\e,\\
v'_{\e}(x) & = \left(
\begin{array}{c}
 \ds \psi _1 \Big( \e \Big[\frac{x' }{\e} \Big]_Y, \frac{x_3}{\delta} \Big)\\[2mm]
 \ds \psi _2 \Big( \e \Big[\frac{x' }{\e} \Big]_Y, \frac{x_3}{\delta} \Big)\\[2mm]
 \ds v_3 \Big(\e \Big[\frac{x' }{\e} \Big]_Y, 0\Big)-{\e\over \d}\Big\{\frac{x' }{\e} \Big\}_Y\cdot \frac{\partial \psi}{\partial X_3} \Big( \e \Big[\frac{x' }{\e} \Big]_Y, \frac{x_3}{\delta} \Big) 
 \end{array} \right)\qquad \hbox{for a.e. } x\in  \O^i_\e,
\end{aligned}$$
where 
\begin{itemize}
\item $\widehat v \in H^1_{per} (Y; \R^3)$,
\item $\psi_{\alpha}  \in {\cal C}^1 (\overline{\omega} ; {\cal C}^3([0,1])), \alpha= 1,2,$ satisfies
\begin{equation*}
\psi_{\alpha}  (x', 0) = v_{\alpha|\Omega^-} (x', 0), \quad \psi_{\alpha}  (x', 1) = v_{\alpha|\Omega^+} (x', 0)\qquad \text{ for every } x'\in \omega,
\end{equation*}
\item $\Psi^{(-)}\in W^{1,\infty}(\Omega^-)$, $\Psi^{(+)}\in W^{1,\infty}(\Omega^+)$ satisfying
$$\Psi^{(\pm)}(x', 0)=0,\quad\hbox{ a.e. in } \omega,\qquad  \Psi^{(-)}=0\quad \hbox{on}\;\; \Gamma,$$
\item $h(x_3)$ is defined as in \eqref{h}.
\end{itemize}

Using \eqref{lem62} we obtain the following convergences:
\begin{gather*}
\begin{aligned}
&\mathcal T''_{\e} (v'_{\e} (\cdot, \cdot\cdot)) \longrightarrow v  \quad \hbox{strongly in } \; L^2(\Omega^{-}; H^1(Y)),\\
&\mathcal T''_{\e} (\nabla v'_{\e} (\cdot, \cdot\cdot)) \longrightarrow \nabla v + \Psi^{(-)}\nabla_y \widehat v \quad \hbox{strongly in } \; L^2(\Omega^{-}\times Y),\\
&\mathcal T''_{\e} (v'_{\e} (\cdot+\d e_3, \cdot\cdot)) \longrightarrow v  \quad \hbox{strongly in } \; L^2(\Omega^{+}; H^1(Y)),\\
&\mathcal T''_{\e} (\nabla v'_{\e} (\cdot+\d e_3, \cdot\cdot)) \longrightarrow \nabla v + \Psi^{(+)}\nabla_y \widehat v \quad \hbox{strongly in } \; L^2(\Omega^{+}\times Y).
\end{aligned}
\end{gather*}
Moreover,
\begin{gather*}
\mathcal T_{\e} (\mathcal T''_{\e} (v'_{\e})) \longrightarrow \left(
\begin{array}{c}
 \ds \psi _1(x', X_3)\\[2mm]
 \ds \psi _2(x', X_3)\\[2mm]
 \ds v_3 (x', 0)
\end{array} 
\right) \quad \hbox{strongly in } \; L^2(\omega; H^1(Y \times B_1)),\\
\frac{\delta^2}{r}\mathcal T_{\e} \left(\mathcal T''_{\e} \left( (\nabla v'_{\e})^{33}_S \right) \right) \longrightarrow - X_1 \frac{\partial^2 \psi_1}{\partial X_3^2} (x', X_3) - X_2 \frac{\partial^2 \psi_2}{\partial X_3^2} (x', X_3) \quad \hbox{strongly in } \; L^2(\omega\times Y \times B_1).
\end{gather*}
Unfolding and passing to the limit we obtain
\begin{multline}
\label{unf-lim-i}
\int_{\Omega^{\pm} \times Y} (\sigma^{\pm} + \widehat{\sigma}^{\pm}) : \Big((\nabla v)_S + \Psi^{(\pm)}(\nabla_y \widehat v)_S \Big) \, dx dy- \frac{\kappa_0^4}{\kappa_1^3}\int_{\omega \times B_1} \Theta : \left( X_1 \frac{\partial^2 \psi_1}{\partial X_3^2} + X_2 \frac{\partial^2 \psi_2}{\partial X_3^2} \right) dx' \, dX = \\
= \int_{\Omega^{\pm}} F \, v \, dx + \int_{\omega \times B_1} (F^m_1 \psi_1 + F^m_2 \psi_2 + F_3 v_3)\, dx' \, dX.
\end{multline}
Since $\sigma^{\pm}$ and $(\nabla v)_S$ do not depend on $y$ and due to the periodicity of the fields $\widehat v$ and $\widehat u^{\pm}$, the above equality reads
\begin{multline*}
\int_{\Omega^{\pm} } \sigma^{\pm} : (\nabla v)_S \, dx +\int_{\Omega^{\pm} \times Y} \widehat{\sigma}^{\pm} :  \Psi^{(\pm)}(\nabla_y \widehat v)_S \, dx dy- \frac{\kappa_0^4}{\kappa_1^3}\int_{\omega \times B_1} \Theta : \left( X_1 \frac{\partial^2 \psi_1}{\partial X_3^2} + X_2 \frac{\partial^2 \psi_2}{\partial X_3^2} \right) dx' \, dX = \\
= \int_{\Omega^{\pm}} F \, v \, dx + \int_{\omega \times B_1} (F^m_1 \psi_1 + F^m_2 \psi_2 + F_3 v_3)\, dx' \, dX.
\end{multline*}

\noindent \textit{Step 3.} To determine $\widehat \sigma$ we first take $v=0$. We obtain
\begin{equation*}
\int_{\Omega^{\pm} \times Y} \widehat{\sigma}^{\pm} :  \Psi^{(\pm)}(\nabla_y \widehat v)_S \, dx dy- \frac{\kappa_0^4}{\kappa_1^3}\int_{\omega \times B_1} \Theta : \left( X_1 \frac{\partial^2 \psi_1}{\partial X_3^2} + X_2 \frac{\partial^2 \psi_2}{\partial X_3^2} \right) dx' \, dX = \int_{\omega \times B_1} (F^m_1 \psi_1 + F^m_2 \psi_2)\, dx' \, dX.
\end{equation*}
Since the right-hand side does not contain $\widehat v$,
$$
\int_{\Omega^{\pm} \times Y} \widehat{\sigma}^{\pm} :  \Psi^{(\pm)}(\nabla_y \widehat v)_S \, dx dy = 0,
$$
which corresponds to the strong formulation
\begin{equation*}
 \left\{
\begin{aligned}
 &\sum_{j = 1}^3 \frac{\widehat \sigma^{\pm}_{ij}}{\partial y_j} = 0, &\text{in } & \Omega^{\pm} \times Y,\\
 &\sum_{j = 1}^3 \widehat \sigma^{\pm}_{ij} = 0, &\text{on } & \partial(\Omega^{\pm} \times Y),
\end{aligned}
\right.
\end{equation*}
for $i = 1, 2, 3$. Therefore, $\widehat \sigma^{\pm} = 0$, and \eqref{unf-lim-i} is rewritten as
\begin{multline}
\int_{\Omega^{\pm} } \sigma^{\pm} : (\nabla v)_S \, dx - \frac{\kappa_0^4}{\kappa_1^3}\int_{\omega \times B_1} \Theta : \left( X_1 \frac{\partial^2 \psi_1}{\partial X_3^2} + X_2 \frac{\partial^2 \psi_2}{\partial X_3^2} \right) dx' \, dX = \\
= \int_{\Omega^{\pm}} F \, v \, dx +  \int_{\omega \times B_1} (F^m_1 \psi_1 + F^m_2 \psi_2 + F_3 v_3)\, dx' \, dX.
\end{multline}
Since the space $W^{1,\infty}(\O^+ ; \R^3)$ is dense in $H^1(\O^+;\R^3)$, the space of functions in $W^{1,\infty}(\O^- , \R^3)$ vanishing on $\Gamma$ is dense in $H^1(\O^-;\R^3)$ and the space ${\cal C}^1 (\overline{\omega} ; {\cal C}^3([0,1]))$ is dense in $L^2(\omega; H^1(0,1))$, the above equality holds for every $v$ in $\V$ and every $ \psi_1$, $\psi_2$ in $L^2(\omega  ; H^1(0,1))$ satisfying 
$$\psi_{\alpha}  (x', 0) = v_{\alpha|\Omega^-} (x', 0), \quad \psi_{\alpha}  (x', 1) = v_{\alpha|\Omega^+} (x', 0)\qquad \text{ for a.e. } x'\in \omega.$$
Finally,  integrating over $D_1$ and due to \eqref{theta-bis} we obtain the result.
\end{proof}

\section{Summarize}
\subsection{Strong formulation}
Strong formulations are the same for the cases $(i)$ and  $(ii)$. We will use the following notation.
\begin{notation}
The convolution of the functions $K$ and $F$ is
$$
(K * \widetilde{F}_{\alpha}^m) (x', X_3) = \int_0^1 K(X_3, y_3) \widetilde{F}_{\alpha}^m (x', y_3) \, dy_3.
$$
\end{notation}
Let $\{\e\}$ be a sequence of positive real numbers which tends to 0. Let $(u_{\e}, \sigma_{\e})$ be the solution of \eqref{wf} and $\widetilde{\mathcal U}_{\e}$ and $\widetilde{\mathcal R}_{\e}$ be the two first terms of the decomposition of $u_{\e}$ in $\Omega_{\e}^i$. Let $f$ satisfy assumptions \eqref{forassump}. Then the limit problems for the cases $(i), (ii)$ can be written as follows.

{\bf Bending problem in the beams}: $(\widetilde{\mathcal U}_1, \widetilde{\mathcal U}_2) \in L^2 (\omega, H^1 (0, 1))^2$ is the unique solution of the problem
\begin{equation}
\label{bend-pr}
 \left\{
\begin{array}{l}
 \ds \frac{\pi \kappa_0^4}{4 \kappa_1^3} E^m \frac{\partial^4 \widetilde{\mathcal U}_{\alpha}}{\partial X_3^4} = \widetilde{F}^m_{\alpha} \quad \text{a.e. in } \omega \times (0, 1),\\
 \ds \frac{\partial \widetilde{\mathcal U}_{\alpha}}{\partial X_3} (\cdot, \cdot, 0) = \frac{\partial\widetilde{\mathcal U}_{\alpha}}{\partial X_3} (\cdot, \cdot, 1)=0, \quad \text{a.e. in } \omega,\\
 \ds  \widetilde{\mathcal U}_{\alpha}(\cdot, \cdot, 0) = u^-_{\alpha|\Sigma},\qquad \widetilde{\mathcal U}_{\alpha}(\cdot, \cdot, 1)= u^+_{\alpha|\Sigma}, \quad \text{a.e. in } \omega,
\end{array}
\right.
\end{equation}

{\bf 3D elasticity problem} in $\Omega^+ \cup \Omega^-$: $(u^{\pm}, \sigma^{\pm}) \in (H^1(\Omega^+ \cup \Omega^-))^3 \times (L^2(\Omega))^{3\times 3}_S$ is the unique weak solution of the problem
\begin{equation}
\label{strlim1}
-\sum_{j = 1}^3 \frac{\partial \sigma^{\pm}_{ij}}{\partial x_j} = F_i \quad \text{in } \Omega^{\pm}, \quad i = 1,2,3,
\end{equation}
together with the boundary conditions
\begin{equation}
\left\{
\begin{aligned}
 &\ds \sigma_{i3}^+ = 0 &\quad &\text{in } \omega \times \{L\},&\\
 &\ds \sigma_{i3}^- = 0 &\quad &\text{in } \omega \times \{-L\},&\\
\end{aligned}
\right.
\end{equation}
and the transmission conditions
\begin{equation}
\label{strlim2}
\left\{
\begin{aligned}
 & [\sigma^{\pm}_{i 3}]_{|\Sigma} = \overline{F}_{i}^m &\quad &\text{on } \Sigma,&\\
 &[u_3^{\pm}]_{|\Sigma} = 0 &\quad &\text{on } \Sigma,&\\
 & \sigma_{\alpha 3}^+ = - \frac{3\pi \kappa_0^4}{\kappa_1^3} E^m [u_{\alpha}^{\pm}]_{|\Sigma} + \int_0^1 K_{\alpha} * \widetilde{F}_{\alpha}^m dX_3 &\quad &\text{on } \Sigma, \quad \alpha = 1, 2.&
\end{aligned}
\right.
\end{equation}

\subsubsection{Derivation of the 3D problem}
\begin{lemma}
The weak formulation of the limit problem can be rewritten as
\begin{multline}
\label{limwf}
\int_{\Omega^+ \cup \Omega^-} \sigma^{\pm} \, : (\nabla v)_S \, dx + \frac{3\pi \kappa_0^4}{\kappa_1^3} E^m \int_{\Sigma} \sum_{\alpha = 1}^2 [u_{\alpha}^{\pm}]_{|\Sigma} [v_{\alpha}^{\pm}]_{|\Sigma} \, ds  = \\
= \int_{\Omega^+ \cup \Omega^-} F \,v \,dx + \int_{\Sigma} \sum_{\alpha = 1}^3 \overline{F}_{\alpha}^m \, v_{\alpha}^- \, ds + \int_{\Sigma} \sum_{\alpha = 1}^2 [v_{\alpha}^{\pm}]_{|\Sigma} \int_0^1 K_{\alpha} * \widetilde{F}_{\alpha}^m dX_3 \, ds, \qquad \forall v\in \V,
\end{multline}
where
\begin{gather*}
 \sigma^{\pm} = \lambda^b (\mathrm{Tr} \, (\nabla u^{\pm})_S) I + 2 \mu^b (\nabla u^{\pm})_S \in L^2 (\Omega^{\pm}; \R^9),\\
 K_{\alpha} (X_3, y_3) = \delta(X_3 - y_3) X_3^2 (3 - 2 X_3) + 6 (1 - 2 X_3) \left( (X_3 - y_3) H(X_3 - y_3) + (1 - y_3)^2 (y_3 - 2 y_3 X_3 - X_3)\right).
\end{gather*}
\end{lemma}
\begin{proof}
\textit{Step 1.} Decomposition of $\widetilde{\mathcal U}_{\alpha}$.

Denote
$$
\begin{aligned}
{\cal V}_d&=\Big\{ \eta\in {\cal C}^3([0,1])\;| \; \eta(X_3)=(b-a)X_3^2(3 - 2X_3)+a,\;\; (a,b)\in \R^2\Big\}.
\end{aligned}
$$
Observe that a function $X_3\longmapsto\eta(X_3)=(b-a)X_3^2(3 - 2X_3)+a$ of ${\cal V}_d$ satisfies
$$\eta(0)=a,\quad \eta(1)=b,\quad {d \eta\over dX_3}(0)=0,\quad {d \eta\over dX_3}(1)=0,\quad \hbox{and}\quad {d^4 \eta\over dX_3^4}=0\enskip \hbox{in}\;\; (0,1).$$
Hence for any function $\psi\in H^2_0(0,1)$ we have
$$\int_0^1{d^2 \eta\over dX_3^2}(t){d^2 \psi\over dX_3^2} (t) \, dt=0.$$
Let $\widetilde{\widetilde{\mathcal U}_\alpha}$ be in $L^2(\omega; H^2_0(0,1))$ the solution of the following problem:
\begin{equation*}
 \left\{
\begin{array}{l}
 \ds \frac{\pi \kappa_0^4}{4 \kappa_1^3} E^m \frac{\partial^4 \widetilde{\widetilde{\mathcal U}_\alpha}}{\partial X_3^4}(x', X_3) = \widetilde{F}^m_{\alpha}(x', X_3) \quad \text{a.e. in } \omega \times (0, 1),\\
 \ds \frac{\partial \widetilde{\widetilde{\mathcal U}_\alpha}}{\partial X_3} (\cdot, \cdot, 0) = \frac{\partial\widetilde{\widetilde{\mathcal U}_\alpha}}{\partial X_3} (\cdot, \cdot, 1)=0, \quad \text{a.e. in } \omega,\\
 \ds  \widetilde{\widetilde{\mathcal U}_\alpha}(\cdot, \cdot, 0) =\widetilde{\widetilde{\mathcal U}_\alpha}(\cdot, \cdot, 1)= 0, \quad \text{a.e. in } \omega.
\end{array}
\right.
\end{equation*}
Using Green's function we can write $\widetilde{\widetilde{\mathcal U}_\alpha}$ in the following way:
\begin{equation*}
 \widetilde{\widetilde{\mathcal U}_\alpha} (x', X_3) = \frac{4 \kappa_1^3}{\pi E^m \kappa_0^4} \int_0^1 \xi_{\alpha} (X_3, y_3) \widetilde{F}_{\alpha}^m (x', y_3) \, dy_3,
\end{equation*}
where $\xi_{\alpha}$ is the solution of the equation
\begin{equation*}
 \left\{
 \begin{aligned}
  &\frac{d^4 \xi_{\alpha}}{d X_3^4} = \delta (X_3 - y_3), \quad y_3 \in (0, 1),\\
  &\frac{d \xi_{\alpha}}{d X_3} (0) = \frac{d \xi_{\alpha}}{d X_3} (1) = 0,\\
  &\xi_{\alpha} (0) = \xi_{\alpha} (1) = 0.
 \end{aligned}
 \right.
\end{equation*}
Solving the above equation we obtain
\begin{equation*}
 \xi_{\alpha} (X_3, y_3) = \frac{1}{6} (X_3 - y_3)^3 H(X_3 - y_3) - \frac{1}{6} (1 - y_3)^2 (2y_3 + 1)X_3^3 + \frac{1}{2}(1 - y_3)^2 y_3 X_3^2,
\end{equation*}
where $H$ is the Heaviside function.

The function $\widetilde{\mathcal U}_\alpha$ is uniquely decomposed as a function belonging to $L^2(\omega; {\cal V}_d)$ and a function\\ in $L^2(\omega; H^2_0(0,1))$
\begin{equation}\label{Eq.65}
\begin{aligned}
\widetilde{\mathcal U}_\alpha(x',X_3)&=(1-X_3)^2(2X_3+1)u^-_{\alpha|\Sigma}(x')-X_3^2(3 - 2X_3)u^+_{\alpha|\Sigma}(x')+\widetilde{\widetilde{\mathcal U}_\alpha}(x',X_3)\\
&=\overline{\mathcal U}_\alpha(x',X_3)+\widetilde{\widetilde{\mathcal U}_\alpha}(x',X_3)\qquad \hbox{for a.e. }(x',X_3)\in \omega\times(0,1).
\end{aligned}
\end{equation}

\textit{Step 2.}
Taking into account decomposition \eqref{Eq.65} and using as a test function $\psi_{\alpha} = [v_{\alpha}^{\pm}]_{|\Sigma} X_3^2 (3 - 2 X_3) + v_{\alpha| \Sigma}^-$ in \eqref{limf} we obtain
\begin{multline}
\int_{\Omega^+ \cup \Omega^-} \sigma^{\pm} \, : (\nabla v)_S \, dx + \frac{3\pi \kappa_0^4}{2\kappa_1^3} E^m \int_{\omega} \sum_{\alpha = 1}^2 \int_0^1 \Big(\frac{\partial^2 \overline{\mathcal U}_1}{\partial X_3^2} [v_1^{\pm}]_{|\Sigma} + \frac{\partial^2 \overline{\mathcal U}_2}{\partial X_3^2}[v_2^{\pm}]_{|\Sigma} \Big) (1 - 2X_3) \, dX_3 \, dx'  =  \\
= \int_{\Omega^+ \cup \Omega^-} F \,v \,dx + \int_{\omega} \sum_{\alpha = 1}^2 [v_{\alpha}^{\pm}]_{|\Sigma} \int_0^1 \Big(\widetilde{F}_{\alpha}^m X_3^2 (3 - 2 X_3) - \frac{3\pi \kappa_0^4}{2 \kappa_1^3} E^m \frac{\partial^2 \widetilde{\widetilde{\mathcal U}_{\alpha}}}{\partial X_3^2}(1 - 2X_3) \Big) dX_3 \, dx'+ \\
+ \int_{\omega} \sum_{\alpha = 1}^3 \overline{F}_{\alpha}^m \, v_{\alpha}^- \, dx'.
\end{multline}
Making use of the solutions for $\overline{\mathcal U}_{\alpha}$ and $\widetilde{\widetilde{\mathcal U}_{\alpha}}$ we can write
\begin{multline}
\int_{\Omega^+ \cup \Omega^-} \sigma^{\pm} \, : (\nabla v)_S \, dx + \frac{3\pi \kappa_0^4}{\kappa_1^3} E^m \int_{\omega} \sum_{\alpha = 1}^2 [u_{\alpha}^{\pm}]_{|\Sigma} [v_{\alpha}^{\pm}]_{|\Sigma} \, dx'  = \int_{\Omega^+ \cup \Omega^-} F \,v \,dx + \\
+ \int_{\omega} \sum_{\alpha = 1}^2 [v_{\alpha}^{\pm}]_{|\Sigma} \int_0^1 \Big( \widetilde{F}_{\alpha}^m X_3^2 (3 - 2 X_3) - 6 (1 - 2 X_3)\int_0^1 \frac{d^2 \xi_{\alpha}}{d X_3^2} (X_3, y_3) \widetilde{F}_{\alpha}^m (x', y_3) dy_3 \Big) dX_3 \, dx'+ \\
+ \int_{\omega} \sum_{\alpha = 1}^3 \overline{F}_{\alpha}^m \, v_{\alpha}^- \, dx'.
\end{multline}
Using the notation for convolution and the expression for $\ds \frac{d^2 \xi_{\alpha}}{d X_3^2}$ we get the result.
\end{proof}
From variational formulation \eqref{limwf} the final strong formulation is obtained.

\subsection{Convergences}
\begin{theorem}
 Under the assumptions \eqref{forassump} on the applied forces, we first have (convergence of the stress energy)
\begin{equation}\label{EQ. 7.5}
\begin{aligned}
\lim_{\e\to 0}{\cal E}(u_\e)  =
 &\int_{\Omega^+ \cup \Omega^-} \sigma^{\pm} \, : (\nabla u)_S \, dx + \frac{\pi \kappa_0^4}{4 \kappa_1^3} E^m \int_{\omega \times (0, 1)} \sum_{\alpha = 1}^2 \Big|\frac{\partial^2 \widetilde{\mathcal U}_{\alpha}}{\partial X_3^2}\Big|^2 \, dx' \, dX_3.
\end{aligned}
\end{equation}
The sequence $(u_{\e}, \sigma_{\e})$ satisfy the following convergences:
\begin{itemize}
 \item $u_{\e} \rightarrow u^-$ strongly in $H^1(\Omega^-)$,\\[2mm]
 $u_{\e}(\cdot + \delta e_3) \rightarrow u^+$ strongly in $H^1(\Omega^+)$,
 \item $\sigma_{\e} \rightarrow \sigma^-$ strongly in $L^2(\Omega^-)$,\\[2mm]
 $\sigma_{\e}(\cdot + \delta e_3) \rightarrow \sigma^+$ strongly in $L^2(\Omega^+)$,
\item $\ds \frac{\d^2}{r} \mathcal T_{\e}' (\sigma_{\e}) \rightarrow \Theta$ strongly in $L^2(\omega \times B_1)$, where
 $$\Theta_{ij} = \left\{
 \begin{aligned}
 &-E^m \Big(X_1\frac{\partial^2 \widetilde{\mathcal U}_1}{\partial X_3^2} + X_2 \frac{\partial^2 \widetilde{\mathcal U}_2}{\partial X_3^2} \Big), &\quad (i, j) = (3, 3),\\
 &0, &\quad \text{otherwise},
 \end{aligned}
 \right.$$
 \item  $\delta \mathcal T_{\e}(\widetilde{\mathcal R}_{\e, \alpha}) \rightarrow \widetilde{\mathcal R}_{\alpha}, \quad \alpha = 1,2$ strongly in $L^2 (\omega; H^1 (0, 1)),$
 \item  $ \mathcal T_{\e} (\widetilde{\mathcal U}_{\e,\alpha}) \rightarrow \widetilde{\mathcal U}_{\alpha}, \quad \alpha = 1,2$ strongly in $L^2(\omega; H^1 (0, 1))$,\\[2mm]
$ \mathcal T_{\e} (\widetilde{\mathcal U}_{\e,3}) \rightarrow u_3^{\pm} (\cdot , 0)$ strongly in $L^2(\omega; H^1(0, 1)),$\\[2mm]
$\ds\frac{\delta}{r}\mathcal T_{\e}(\widetilde{\mathcal U}_{\e,3} - \widetilde{\mathcal U}_{\e,3} (\cdot, \cdot, 0)) \rightarrow 0$ strongly in $L^2 (\omega; H^1 (0, 1))$.
\end{itemize}
\end{theorem}
\begin{proof}
\textit{Step 1.} We prove \eqref{EQ. 7.5}.

We first recall the classical identity: if $T$ is a symmetric $3\times 3$ matrix we have
\begin{equation}\label{EQ.770}
\begin{aligned}
\lambda^m Tr(T) Tr(T) + \sum_{i,j=1}^3 2\mu^m T_{ij} T_{ij} &= E^m T_{33}^2 +  \frac{E^m}{(1+\nu^m)(1-2\nu^m)}(T_{11}+T_{22}+2\nu^m T_{33})^2\\
 &+\frac{E^m}{2(1+\nu^m)}[(T_{11}-T_{22})^2 + 4(T_{12}^2+T_{13}^2+T_{23}^2)].
\end{aligned}
\end{equation}
Now, we consider the total elastic energy of displacement $u_\e$ given by  \eqref{(2.9)}
\begin{equation}\label{EQ.77}
{\cal E}(u_\e)=\int_{\Omega_{\e}} \sigma_{\e} : (\nabla u_\e)_S \, dx = \int_{\Omega_{\e}} f_{\e}\cdot u_\e\,dx.
\end{equation}
The left hand side of \eqref{EQ.77} is
\begin{equation}\label{EQ.78}
\begin{aligned}
{\cal E}(u_\e)&=\int_{\Omega_{\e}} \sigma_{\e} : (\nabla u_\e)_S\, dx \\
& = \int_{\Omega^{-}} \sigma_{\e} : (\nabla u_\e)_S dx +\int_{\Omega^{i}_{\e}} \sigma_{\e} : (\nabla u_\e)_S dx +\int_{\Omega^{+}_{\e}} \sigma_{\e} : (\nabla u_\e)_S dx \\
& = \int_{\Omega^{-}} \sigma_{\e} : (\nabla u_\e)_S dx +{r^2\d\over \e^2}\int_{\omega\times B_1} {\cal T}'_\e(\sigma_{\e}) : {\cal T}'_\e\big((\nabla u_\e)_S\big) dx' dX +\int_{\Omega^{+}_{\e}} \sigma_{\e} : (\nabla u_\e)_S dx. 
\end{aligned}
\end{equation}
The second term of the right hand side of the above equation is transformed using the identity \eqref{EQ.770}
$$
\begin{aligned}
&\int_{\omega\times B_1} {\cal T}'_\e(\sigma_{\e}) : {\cal T}'_\e\big((\nabla u_\e)_S\big) dx' dX\\
 = &\int_{\omega\times B_1} \Big( E^m ((\nabla u_\e)_S^{33})^2 +  \frac{E^m}{(1+\nu^m)(1-2\nu^m)} \Big((\nabla u_\e)_S^{11}+ (\nabla u)_S^{22}+2\nu^m (\nabla u_\e)_S^{33} \Big)^2  \\
+ &\frac{E^m}{2(1+\nu^m)} \left[\Big((\nabla u_\e)_S^{11} - (\nabla u_\e)_S^{22} \Big)^2+ 4 \Big(((\nabla u_\e)_S^{12})^2 + ((\nabla u_\e)_S^{13})^2 + ((\nabla u_\e)_S^{23})^2 \Big) \right] \Big) dx'dX.
\end{aligned}
$$ Then by standard weak lower-semi-continuity, Lemma \ref{symmX} and  \eqref{theta} give (we recall that $\widetilde{\cal R}_3=0$)
\begin{equation}
\label{inf}
 \int_{\Omega^+ \cup \Omega^-} \sigma^{\pm} \, : (\nabla u)_S \, dx +\frac{\kappa_0^4}{\kappa_1^3} \int_{\omega\times B_1}  E^m ((\nabla u)_S^{33})^2 \, dx'dX \leq \liminf_{\e \to 0} {\cal E}(u_\e).
\end{equation}
Besides, the convergences in Proposition \ref {prop41} lead to
$$
\begin{aligned}
\limsup_{\e \to 0}{\cal E}(u_\e)&=\limsup_{\e \to 0} \int_{\Omega_{\e}} f_{\e}\cdot u_\e\,dx=\lim_{\e \to 0} \int_{\Omega_{\e}} f_{\e}\cdot u_\e\,dx\\
&=\int_{\Omega^+ \cup \Omega^-} F \cdot u \,dx + \int_{\omega \times (0, 1)} \sum_{\alpha = 1}^2 \widetilde{F}_{\alpha}^m \widetilde{\cal U}_{\alpha} dx' \, dX_3 + \int_{\omega} \overline{F}_3^m u_3 dx'.
\end{aligned}
$$
Hence
$$ \limsup_{\e \to 0}{\cal E}(u_\e) \leq \liminf_{\e \to 0} {\cal E}(u_\e).$$
Therefore (we recall that $\widetilde{\cal U}'_3=0$),
\begin{equation}\label{Conv711}
\lim_{\e \to 0}{\cal E}(u_\e) = \int_{\Omega^+ \cup \Omega^-} \sigma^{\pm} \, : (\nabla u)_S \, dx + \frac{\pi \kappa_0^4}{4 \kappa_1^3} E^m \int_{\omega \times (0, 1)} \sum_{\alpha = 1}^2 \Big|\frac{\partial^2 \widetilde{\mathcal U}_{\alpha}}{\partial X_3^2}\Big|^2 \, dx' \, dX_3.
\end{equation}
\noindent \textit{Step 2.}  As immediate consequence of the above convergence \eqref{Conv711} we have 
\begin{gather}
\label{conv-st}
\begin{aligned}
&\sigma_{\e} \rightarrow \sigma^- \quad \text{ strongly in }L^2(\Omega^-),\\
&\sigma_{\e}(\cdot + \delta e_3) \rightarrow \sigma^+ \quad \text{ strongly in }L^2(\Omega^+),\\
&\frac{\delta^2}{r} \mathcal T_{\e}^{'} (\sigma_{\e}) \rightarrow \Theta, \quad \text{ strongly in } L^2 (\omega \times B_1).
\end{aligned}
\end{gather}
Hence
\begin{gather}
\label{conv-st-bis}
\begin{aligned}
&(\nabla u_{\e})_S \rightarrow (\nabla u^-)_S \quad \text{ strongly in }L^2(\Omega^-),\\
&(\nabla u_{\e} (\cdot + \delta e_3))_S \rightarrow (\nabla u^+)_S \quad \text{ strongly in }L^2(\Omega^+),\\
&\frac{\delta^2}{r} \mathcal T_{\e}^{'} ((\widetilde{\nabla u_{\e}})_S) \rightarrow X \quad \text{ strongly in } L^2 (\omega \times B_1).
\end{aligned}
\end{gather}
From \eqref{conv-st-bis}$_1$ and the Korn inequality  \eqref{Eq19} in $\O^-$ we deduce that
$$u_{\e} \rightarrow u^-\;\;\hbox{ strongly in }\;\; H^1(\Omega^-).$$
The above strong convergence and estimates \eqref{estu1}-\eqref{estu2} yield
\begin{equation}\label{EQ.710}
u_{\e}(\cdot, 0)1_{\widehat{\omega}_\e}\rightarrow u^-_{|\Sigma}=\widetilde{\cal U}(\cdot, 0)\;\; \hbox{strongly in } L^2(\Sigma).
\end{equation}
From \eqref{conv-st-bis}$_3$ and \eqref{beams-grad}$_4$ we derive that
\begin{multline*}
 \frac{\delta}{r}\frac{\partial \mathcal T_{\e}(\widetilde{\mathcal U}_{\e,3})}{\partial X_3} + \delta\frac{\partial \mathcal T_{\e}(\widetilde{\mathcal R}_{\e,1})}{\partial X_3} X_2  - \delta\frac{\partial \mathcal T_{\e}(\widetilde{\mathcal R}_{\e,2})}{\partial X_3} X_1 + \frac{\delta^2}{r} \mathcal T_{\e}^{'}((\widetilde{\nabla \bar u_{\e}})_S)_{33} \rightarrow  \frac{\partial \widetilde{\mathcal R}_1}{\partial X_3} X_2  - \frac{\partial \widetilde{\mathcal R}_2}{\partial X_3} X_1 \\ \text{strongly in } L^2(\omega \times B_1).
\end{multline*}
Hence, equalities \eqref{prop-warp}$_1$-\eqref{prop-warp}$_3$ with estimates \eqref{Diff}$_1$ - \eqref{diffu1} and convergence \eqref{EQ.710} lead to
\begin{gather}
\label{conv-r}
 \delta \mathcal T_{\e}(\widetilde{\mathcal R}_{\e, \alpha}) \rightarrow \widetilde{\mathcal R}_{\alpha}, \quad \alpha = 1,2 \quad \text{ strongly in } L^2 (\omega; H^1 (0, 1)),\\
 \frac{\delta}{r}\mathcal T_{\e}(\widetilde{\mathcal U}_{\e,3} - \widetilde{\mathcal U}_{\e,3} (\cdot, \cdot, 0)) \rightarrow 0 \quad \text{ strongly in } L^2 (\omega; H^1 (0, 1)),\\
\label{conv-du}
\mathcal T_{\e}(\widetilde{\mathcal U}_{\e,3}) \rightarrow u^{-}_{3|\Sigma} =u^{+}_{3|\Sigma} \quad \text{ strongly in } L^2 (\omega , H^1(0, 1)).
\end{gather}
The fourth estimate in Lemma \ref{unfest}, the convergences \eqref{EQ.710}-\eqref{conv-r} and the equalities \eqref{uc6} imply that
\begin{equation}\label{EQ.718}
 \mathcal T_{\e} (\widetilde{\mathcal U}_{\e,\alpha}) \rightarrow \widetilde{\mathcal U}_{\alpha}, \quad \alpha = 1,2 \quad \text{ strongly in } L^2(\omega; H^1 (0, 1)).
\end{equation}
From convergences \eqref{conv-du}-\eqref{EQ.718} and estimates \eqref{diffu1}-\eqref{diffu2} we obtain
\begin{equation}\label{EQ.719}
u_{\e}(\cdot, \delta)1_{\widehat{\omega}_\e}\rightarrow u^+_{|\Sigma}=\widetilde{\cal U}(\cdot, 1)\;\; \hbox{strongly in } L^2(\Sigma).
\end{equation}
Finally, due to the above strong convergence together with \eqref{conv-st-bis}$_2$ and \eqref{estset0}$_2$, we get  $u_{\e}(\cdot + \delta e_3) \rightarrow u^+$ strongly in $H^1(\Omega^+)$. 
\end{proof}

\section{Complements}

The case
\begin{equation*}
r = \kappa_1 \e^2, \quad \delta = \kappa_2 \e^2, \quad \kappa_1, \kappa_2 > 0,
\end{equation*}
can also be considered, but should be studied separately. The structure obtained will no longer correspond to the set of the thin beams but to some kind of  perforated domain.
%

\section{Appendix}
\noindent Let $\chi$ be in ${\cal C}^\infty_c(\R^2)$ such that $\chi(y)=1$ in   $D_1$.

\begin{lemma}\label{lem61} Let $\phi$ be in $W^{1,\infty}(\omega)$ and $\phi_{\e,r}$ defined by
$$
\phi_{\e,r}(x')=\chi\Big({\e\over r}\Big\{{x'\over \e}\Big\}_Y\Big)\phi\Big(\e\Big[{x'\over \e}\Big]_Y\Big)+\Big[1-\chi\Big({\e\over r}\Big\{{x'\over \e}\Big\}_Y\Big)\Big]\phi(x')\quad \hbox{for a.e. }\;x'\in \omega.
$$
If $\ds {r\over \e}\to 0$ then for every $p\in [1,+\infty)$ we have
$$
\phi_{\e,r}\longrightarrow \phi\quad \hbox{strongly in } \; W^{1,p}(\omega).
$$ 
\end{lemma}
\begin{proof}
For the sake of simplicity we extend $\phi$ in a function belonging to $W^{1,\infty}(\R^2)$ still denoted $\phi$. We denote
$$\widetilde{\Xi}_\e=\Big\{\xi\in \Z^2\;;\; (\e \xi+\e Y)\cap \omega\not= \emptyset\Big\}.$$  Observe that $\Xi_\e \subset \widetilde{\Xi}_\e$. Consider the following estimate:
\begin{equation}\label{Linf}
\begin{aligned}
 \| \phi_{\e,r} - \phi \|_{L^\infty(\omega)} = & \left\| \chi\Big({\e\over r}\Big\{{\cdot\over \e}\Big\}_Y\Big) \left( \phi\Big(\e\Big[{\cdot\over \e}\Big]_Y \Big) - \phi \right) \right\|_{L^\infty (\omega)} \leq  \sup_{\xi \in \widetilde{\Xi}_\e} \left\| \chi\Big({\cdot\over r}\Big) ( \phi(\e\xi) - \phi(\e \xi + \cdot)) \right\|_{L^\infty (Y_{\e})} \\ 
= & \sup_{\xi \in \widetilde{\Xi}_\e} \left\| \chi\Big({\e \over r} \cdot\Big) ( \phi(\e\xi) - \phi(\e \xi + \e \cdot)) \right\|_{L^\infty (Y)} \leq  \e \|\chi\|_{L^\infty(\R^2)}\|\nabla \phi\|_{L^\infty(\R^2)}.
\end{aligned}
\end{equation}
The partial derivative of $\phi_{\e,r} - \phi$ with respect to $x_\alpha$ is 
$$
\begin{aligned}
&\frac{\partial(\phi_{\e,r} - \phi)}{\partial x_{\alpha}} (x') ={1\over r}  \frac{\partial \chi}{\partial X_{\alpha}}\Big({\e\over r}\Big\{{x'\over \e}\Big\}_Y\Big) \left( \phi\Big(\e\Big[{x'\over \e}\Big]_Y \Big) - \phi(x')\right) - \chi\Big({\e\over r}\Big\{{x'\over \e}\Big\}_Y\Big) \frac{\partial \phi}{\partial x_{\alpha}}(x'),\quad \hbox{for a.e. } x'\in \omega,\\
&\frac{\partial(\phi_{\e,r} - \phi)}{\partial x_{\alpha}} (\e \xi+\e y') ={1\over r}  \frac{\partial \chi}{\partial X_{\alpha}}\Big({\e\over r} y' \Big) \left( \phi(\e \xi) - \phi(\e \xi+\e y')\right) - \chi\Big({\e\over r}y'\Big) \frac{\partial \phi}{\partial x_{\alpha}}(\e \xi+\e y'),\quad \xi\in \widetilde{\Xi}_\e,\;\; \hbox{for a.e. } y'\in Y.
\end{aligned}
$$
\noindent Since $\chi$ has a compact support in $\R^2$, there exists $R>0$ such that $\hbox{supp}(\chi)\subset D_R$. Thus, the support of the function $y'\longmapsto \ds  \chi\Big({\e\over r}y'\Big)$ is included in the disc $D_{rR/\e}$. As a consequence we get for a.e. $y'\in D_{rR/\e}$
$$\left| \phi(\e \xi) - \phi(\e \xi+\e y')\right|\le r R\|\nabla\phi\|_{L^\infty(\R^2)}.$$ 
 Using the above estimate, for the norms of the derivatives we first have
$$
\begin{aligned}
 \left\| \frac{\partial(\phi_{\e,r} - \phi)}{\partial x_{\alpha}} \right\|^p_{L^p (\e\xi+\e Y)} &=\e^{2} \left\|\frac{\partial \chi}{\partial X_{\alpha}}\Big({\e \over r}\cdot\Big){ \phi (\e\xi) - \phi(\e\xi+\e \cdot)\over r} - \chi\Big({\e\over r}\cdot \Big) \frac{\partial \phi}{\partial x_{\alpha}} (\e\xi+\e\cdot) \right\|^p_{L^p (Y)} \\
&\leq  Cr^{2} \|\nabla\chi\|^p_{L^\infty(\R^2)}\|\nabla\phi\|^p_{L^\infty(\R^2)}.
\end{aligned}
$$
The constant does not depend on $\e$ and $r$. Combining the above estimates for $\xi\in \widetilde{\Xi}_\e$, that gives 
\begin{equation}\label{GLinf}
 \left\| \nabla (\phi_{\e,r} - \phi) \right\|_{L^p (\omega)} \leq  C \Big({r \over \e}\Big)^{2/p}\|\nabla\chi\|_{L^\infty(\R^2)}\|\nabla\phi\|_{L^\infty(\R^2)}.
 \end{equation}
The constant does not depend on $r$ and $\e$. 
Hence, estimates \eqref{Linf} and \eqref{GLinf} imply that $\phi_\e$ strongly converges toward $\phi$ in $W^{1,p}(\omega)$.
\end{proof}

\section{Simulation results}
In this section solutions $u_{r, \e, \d}$ of the equation \eqref{strf} are compared with the solution $u$ of \eqref{strlim1}--\eqref{strlim2} for the case $n = 2$. The solutions $u_{r, \e, \d}$ are computed numerically for different $r, \e, \d$ with the commercial finite element software COMSOL Multiphysics. The relation between the parameters is chosen in a following way
$$ r = \e^{3/2}, \qquad \d = \e^{4/3}, $$
what corresponds to the Case~$(ii)$ with $\eta_0 = 1.5, \kappa_0 = \kappa_1 = 1$. Comparison between sequence of the solutions $u_{\e}$ and $u$ is done for jumps in displacement and stress. Components of the jumps are computed for different $\e$ and it is shown that the following norms tend to 0 as $\e$ tends to 0:
$$ \| [u_{\e,1}] - [u_1] \|_{L^2 (\Sigma)}, \quad \| [u_{\e, 2}] \|_{L^2 (\Sigma)}, \quad \| \sigma^+_{\e,12} - \sigma^+_{12} \|_{L^2 (\Sigma)}, \quad \| \sigma^+_{\e,22} - \sigma^+_{22} \|_{L^2 (\Sigma)}.$$
The stiffness coefficients and applied force are chosen as follows
$$ E = 2\cdot 10^{11}, \qquad \nu = 0.3, \qquad f_{\e} = (10^3, 10^3).$$
 \begin{figure}[h!]
\centering
\label{sim1}
  \subfigure[Stresses $\sigma_{eq}^+, \sigma_{eq}^-$ in the regions $\Omega^+ \cup \Omega^-$]{\includegraphics[width = 0.4\textwidth]{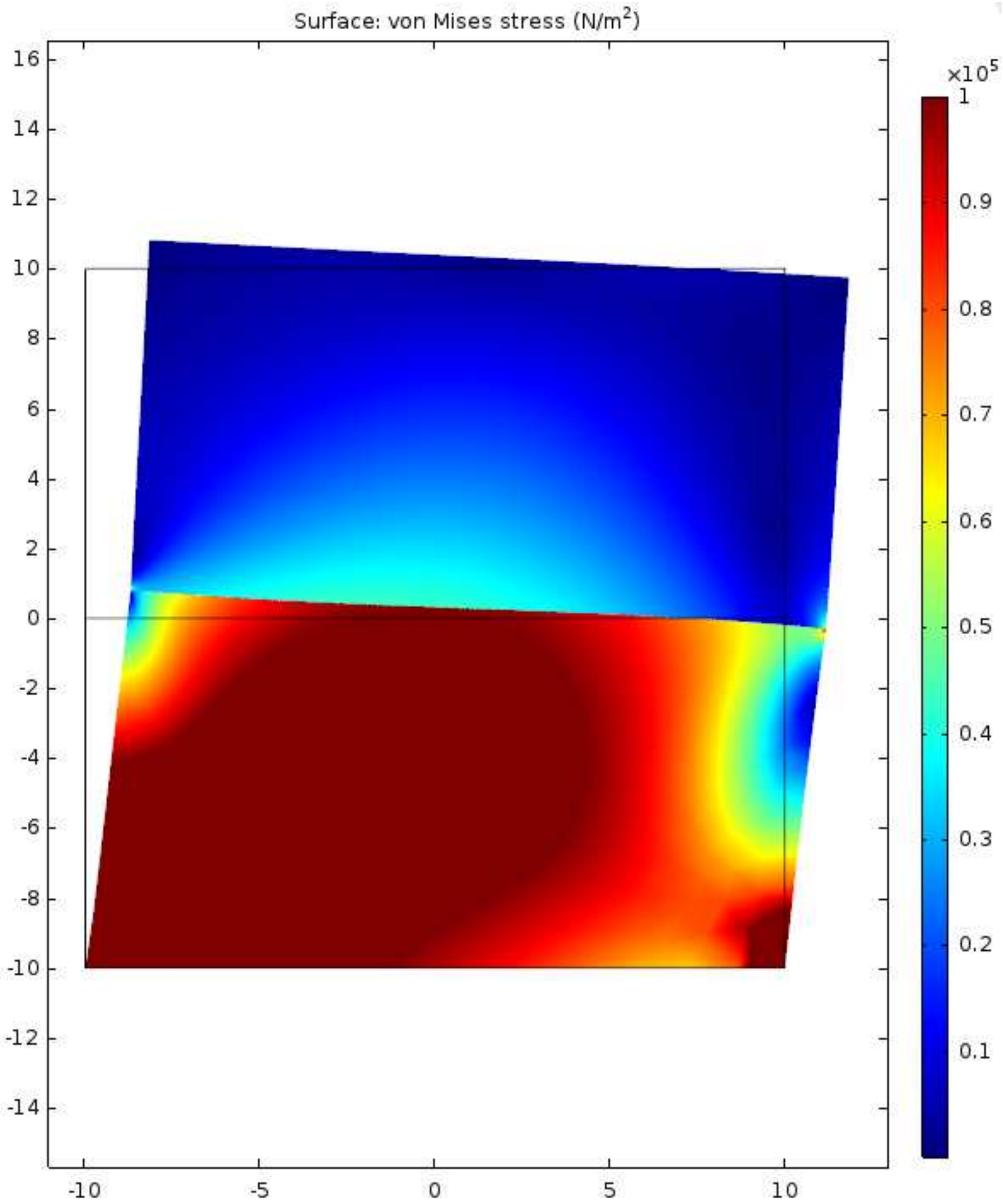}} \hspace{2cm}
  \label{sim2}
  \subfigure[Stresses $\sigma_{\e, eq}$ in the layer $\Omega_{\e}^i$]{\includegraphics[width = 0.4\textwidth]{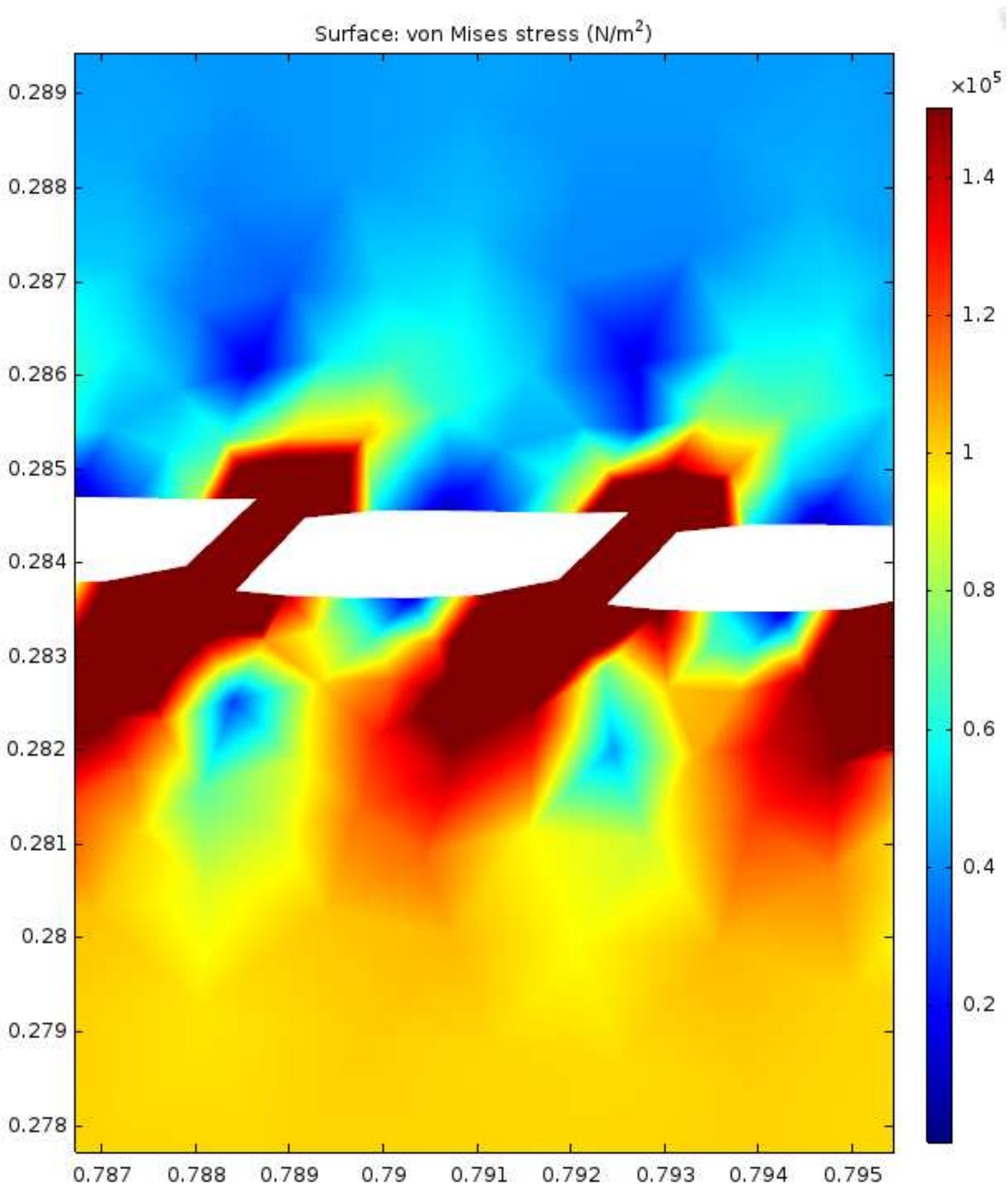}}
\caption{Simulation results}
\label{sim}
 \end{figure}
 
The equivalent von Mises stresses for macro-- and (local) micro--problems are presented in Fig.~\ref{sim}. Fig.~\ref{sim}~(a) provides the solution of the equation \eqref{strlim1}--\eqref{strlim2} in macroscopic blocks and jumps in the equivalent von Mises stresses across the interface can be observed. Fig.~\ref{sim}~(b) shows the local $\e$--solution in the layer for $\e = 0.004$.

Comparison results for chosen $\e$ are gathered in the Table \ref{conv-sim}.
 
  \begin{table}[h!]
  \centering
 \begin{tabular}{|c|c|c|c|c|}
 \hline
 $\e$ & {\small $\| [u_{\e,1}] - [u_1] \|$} &  {\small $\| [u_{\e, 2}] \|$} &  {\small $\| \sigma^+_{\e,12} - \sigma^+_{12} \|$} &  {\small $\| \sigma^+_{\e,22} - \sigma^+_{22} \|$} \\ \hline
 $0.1$ & $2.9 \cdot 10^{-5}$ & $3 \cdot 10^{-5}$ & 1.05 & 1.1 \\\hline
 $0.02$ & $2.5 \cdot 10^{-5}$ & $1.6 \cdot 10^{-5}$ & 0.59 & 0.6 \\\hline
 $0.004$ & $7.6 \cdot 10^{-6}$ & $7.7 \cdot 10^{-6}$ & 0.3 & 0.1 \\\hline
 \end{tabular}
   \caption{Norms of the residual}
  \label{conv-sim}
 \end{table}

{\bf Acknowledgements.} This work was supported by Deutsche Forschungsgemeinschaft (Grants No. OR 190/4--2 and OR 190/6--1).

\end{document}